\documentclass{amsart}

\usepackage{stmaryrd}

\usepackage{amsmath, amsthm, amssymb}

\theoremstyle{definition}

\usepackage{scalerel}
\usepackage{mathrsfs}

\usepackage{tikz}
\usepackage{hyperref}
\usepackage{placeins}

\usepackage{proof}

\usepackage{verbatim}
\usepackage{mathtools}
\usepackage{cleveref}

\usepackage{bbm}

\usepackage[bottom]{footmisc}

\usepackage{enumitem}

\usepackage{biblatex} 
\newtheorem{theorem}{Theorem}
\newtheorem{definition}[theorem]{Definition}
\newtheorem{corollary}{Corollary}[theorem]
\newtheorem{prop}[theorem]{Proposition}
\newtheorem{lemma}[theorem]{Lemma}

\newtheorem{question}[theorem]{Open Question}
\newtheorem{remark}[theorem]{Remark}
\newtheorem{claim}{Claim}[theorem]

\newenvironment{claimproof}[1]{\par\noindent\textit{Proof of the Claim.}\space#1}{\hfill $\blacksquare$}

\newcommand\A{\mathscr{A}}
\newcommand\I{\mathscr{I}}

\newcommand\ZFUR{\textup{ZFU}_\text{R}}
\newcommand\ZFCUR{\textup{ZFCU}_\text{R}}

\renewcommand{\P}{{\mathbb P}}

\newcommand{\barx}{\bar{x}}

\newcommand{\dotx}{\dot{x}}
\newcommand{\doty}{\dot{y}}
\newcommand{\dotz}{\dot{z}}
\newcommand{\dotu}{\dot{u}}
\newcommand{\dotv}{\dot{v}}
\newcommand{\dotw}{\dot{w}}
\newcommand{\forces}{\Vdash}
\newcommand{\nforces}{\nVdash}

\def\<#1>{\left\langle#1\right\rangle}
\def\[#1]{\left\llbracket#1\right\rrbracket}
\renewcommand{\restriction}{\mathord{\upharpoonright}}
\newcommand{\Aeq}{\overset{\mathscr{A}}{=}}

\newcommand{\xset}{\dot{x}^{\textup{Set}}}
\newcommand{\yset}{\dot{y}^{\textup{Set}}}

\title{\textbf{Axiomatization and Forcing in Set Theory with Urelements}}
\author{Bokai Yao}\address[Bokai Yao]{
Peking University}
 \email{bokaiyao1@gmail.com}
\urladdr{https://bokaiyao.com}

\thanks{This paper is part of the author's doctoral dissertation completed at the University of Notre Dame. The author would like to thank Joel David Hamkins for providing insightful feedback on the earlier drafts of this paper. The author is also grateful to the anonymous referee for their extremely detailed and helpful comments.}

\usepackage{graphicx}

\begin{document}
\maketitle

\begin{abstract}
In the first part of this paper, we consider several natural axioms in urelement set theory, including the Collection Principle, the Reflection Principle, the Dependent Choice scheme and its generalizations, as well as other axioms specifically concerning urelements. We prove that these axioms form a hierarchy over $\ZFCUR$ (ZFC with urelements formulated with Replacement) in terms of direct implication. The second part of the paper studies forcing over countable transitive models of $\ZFUR$. We propose a new definition of $\P$-names to address an issue with the existing approach. We then prove the fundamental theorem of forcing with urelements regarding axiom preservation. Moreover, we show that forcing can destroy and recover certain axioms within the previously established hierarchy. Finally, we demonstrate how ground model definability may fail when the ground model contains a proper class of urelements.
\end{abstract}

\section{Introduction}
This paper explores two related topics in set theory with urelements: axiomatization and forcing. One interesting feature of urelement set theory is that once a proper class of urelements is allowed, many ZF-theorems, such as the Collection Principle and the Reflection Principle, are no longer provable. In Section \ref{Section2}, we show that these principles, together with other axioms concerning urelements, form a hierarchy in terms of direct implication over the theory $\ZFCUR$ (Theorem \ref{maintheorem1}). Consequently, the Collection Principle and the Reflection Principle can be characterized by the arrangement of urelements (Corollary \ref{urcharacterization}). We apply these results to show that the Collection Principle is equivalent to the \L o\'s theorem for internal ultrapowers (Theorem \ref{thm:collection<->losthm}).

 We turn to forcing with urelements in Section \ref{section3}. This topic has been studied in Blass and {\v{S}}{\v{c}}edrov ( \cite{blass1989freyd}) and Hall (\cite{hall2002characterization} and \cite{Hall2007-ERIPMA}), yet two main issues remain unexplored in the literature. The first issue concerns the forcing machinery in the presence of urelements. Previous studies have treated each urelement as a distinct copy of the empty set, making it its own $\P$-name. This approach, however, causes the forcing relations to lack a desired property called \textit{fullness}. A new forcing machinery is thus proposed (Definition \ref{newpnames}). We show that in any countable transitive model $M$ of $\ZFCUR$, every new forcing relation in $M$ is full if and only if $M$ satisfies the Collection Principle (Theorem \ref{collection<->fullness}). This, along with Theorem \ref{thm:collection<->losthm}, suggests that a robust axiomatization of ZFC with urelements should include the Collection Principle.

The second issue is the interaction between forcing and the hierarchy of axioms isolated in Section \ref{Section2}. Existing studies typically assume that the ground model contains only a set of urelements, which makes the preservation of many axioms trivial. We show that forcing with a class of urelements can \textit{preserve}, \textit{destroy}, and \textit{recover} various axioms within the hierarchy. For instance, forcing preserves Replacement (Theorem \ref{forcingpreservesreplacement}) but not the DC$_{\omega_1}$-scheme (Theorem \ref{thm:ForcingDesotryDCK}); the Reflection Principle can be recovered by forcing from models of $\ZFCUR$ + DC$_\omega$-scheme (Theorem \ref{recovercollection}). Finally, we show how ground model definability may fail when there is a proper class of urelements.

The rest of this section introduces our notations and reviews some basic facts about urelement set theory. Urelements, which are sometimes also called ``atoms'', are members of sets that are not themselves sets. The language of urelement set theory, in addition to $\in$, contains a unary predicate $\A$ for urelements. $Set(x)$ abbreviates $\neg\A(x)$. The standard axioms (schemes) of ZFC, modified to allow urelements, are as follows.

\begin{itemize}
\item[] (Axiom $\A$) $\forall x (\A(x) \rightarrow \neg \exists y (y \in x))$.
\item[] (Extensionality) $\forall x, y (Set(x) \land Set(y) \land \forall z (z \in y \leftrightarrow z \in x) \rightarrow x = y)$
\item[] (Foundation) $\forall x (\exists y (y \in x) \rightarrow \exists z\in x \ (z \cap x = \emptyset))$
\item[] (Pairing) $\forall x, y \exists z \forall v (v \in z \leftrightarrow v = x \lor v = y )$
\item[] (Union) $\forall x \exists y \forall z (z \in y \leftrightarrow \exists w \in x \ ( z \in w))$.
\item[] (Powerset) $\forall x \exists y \forall z (z \in y \leftrightarrow Set(z) \land z \subseteq x)$
\item [] (Separation) $\forall x, u \exists y \forall z (z \in y \leftrightarrow z \in x \land \varphi(z, u))$
\item[] (Infinity) $\exists s (\exists y \in s \ (Set(y) \land \forall z (z \notin y)) \land \forall x \in s \ (x \cup \{x\} \in s))$
\item[] (Replacement) $\forall w, u (\forall x \in w \ \exists ! y \varphi(x, y, u)   \rightarrow \exists v \forall x \in w \ \exists y \in v \ \varphi(x, y, u))$
\item[] (AC) Every set is well-orderable.
\end{itemize}

\begin{definition}\label{def:ZUandZFU}
\ \newline ZU = Axiom $\A$ + Extensionality + Foundation + Pairing + Union + Powerset + Infinity + Separation.\\
$\ZFUR = $ ZU + Replacement. \\
$\ZFCUR = $ $\ZFUR$ + AC.
\end{definition}
\noindent In $\ZFUR$, every object $x$ has a \textit{kernel}, denoted by $ker(x)$, which is the set of the urelements in the transitive closure of $\{x\}$. The kernel of a urelement is then its singleton, which is somewhat nonstandard but will be useful for our purpose. Note that $x \subseteq y$ is simply $\forall z \in x (z \in y)$, so the power set of a set $x$ is $P(x) =\{y \mid Set(y) \land y \subseteq x\}$. A set is pure if its kernel is empty. $V$ denotes the class of all pure sets. $Ord$ is the class of all ordinals, which are transitive \textit{pure} sets well-ordered by the membership relation. For any given \textit{set} of urelements $A$, the $V_{\alpha}(A)$-hierarchy is defined as usual, i.e.,
\begin{itemize}
    \item [] $V_0(A) = A$;
    \item [] $V_{\alpha+1}(A) = P(V_{\alpha}(A)) \cup A$;
    \item [] $V_{\gamma}(A) = \bigcup_{\alpha < \gamma} V_\alpha(A)$, where $\gamma$ is a limit;
    \item [] $V(A) = \bigcup_{\alpha \in Ord} V_\alpha(A)$.
\end{itemize}
\noindent Note that at each successor stage we must include $A$ because $P(V_{\alpha}(A))$ only contains sets.

We will also let $\A$ stand for the class of all urelements. $A \subseteq \A$ thus means ``$A$ is a set of urelements'', i.e., $Set(A) \land \forall x \in A \ (\A(x))$. For every $x$ and $A \subseteq \A$, $x \in V(A)$ if and only if $ker(x) \subseteq A$. $U$ denotes the class of all objects, i.e., $U = \bigcup_{A\subseteq\A} V(A)$. Every permutation $\pi$ of a set of urelements can be extended to a definable permutation of $\A$ by letting $\pi$ be identity elsewhere; by well-founded recursion, $\pi$ can be further extended to a permutation of $U$ by letting $\pi x$ be $\{\pi y : y \in x \}$ for every set $x$. $\pi$ preserves $\in$  and is thus an automorphism of $U$. For every $x$ and automorphism $\pi$, $\pi$ point-wise fixes $x$ whenever $\pi$ point-wise fixes $ker(x)$. 

Let ``$\A$ is a set'' abbreviate the axiom 
\begin{itemize}
\item [] $\exists x (Set(x) \land \forall y (y \in x \leftrightarrow \A(y))).$
\end{itemize}
It is consistent with $\ZFCUR$ that $\A$ is a proper class (i.e., $\neg (\A \text{ is a set})$ ). However, in urelement set theory proper classes can be rather ``small''. For example, $\ZFCUR$ has models in which $\A$ is a proper class but every set of urelements is finite; consequently, $\ZFCUR$ cannot prove the Collection Principle (this will be discussed in length in Section \ref{section2.5}).
\begin{itemize}
   \item[] (Collection) $\forall w, u (\forall x \in w \exists y \varphi(x, y, u)   \rightarrow \exists v \forall x \in w \exists y \in v \varphi(x, y, u))$.
\end{itemize}
Collection is provable in ZF without urelements and sometimes viewed as part of the axiomatization of ZF, which is why the subscript $R$ is added to $\ZFCUR$. However, Collection cannot exclude models with small proper classes, e.g., $\ZFCUR$ + Collection has models where $\A$ is a proper class but every set of urelements is countable (Theorem \ref{zfcurindependece} (1)).  We end this section with a useful fact that $\ZFUR$ proves a restricted version of Collection.
\begin{prop}[$\ZFUR$]\label{weakcollection}
\
\begin{itemize}
   \item[] $\forall w, u, A \subseteq \A (\forall x \in w \exists y \in V(A) \varphi(x, y, u)   \rightarrow \exists v \forall x \in w \exists y \in v \varphi(x, y, u))$.
\end{itemize}
Hence, $\A$ is a set $\rightarrow$ Collection.
\end{prop}
\begin{proof}
For every $x \in w$, let $\alpha_x$ be the least $\alpha$ such that there is some $y \in V_\alpha (A)$ with $\varphi(x, y, u)$ and let $\alpha = \bigcup_{x \in w} \alpha_x$. $V_\alpha (A)$ is a desired collection set $v$.
\end{proof}

\section{A hierarchy of axioms}\label{Section2}
\subsection{Reflection}
In ZF set theory, the reflection principle is normally formulated as the L\'evy-Montague reflection. Namely,
\begin{itemize}
    \item [] $\forall \alpha \exists \beta > \alpha \forall v_1, ..., v_n \in V_{\beta} (\varphi(v_1, ..., v_n) \leftrightarrow \varphi^{V_\beta}(v_1, ..., v_n))$.
\end{itemize}
\noindent This form of reflection, however, cannot hold when there is a proper class of urelements since no $V_\alpha(A)$ can reflect the statement that $\A$ is a proper class. Thus, with urelements the reflection principle should be formulated in a more general way as follows.
\begin{itemize}
    \item [] (RP) $\forall x \exists t (x \subseteq t \land t \text{ is transitive} \land \forall v_1, ..., v_n\in t (\varphi(v_1, ..., v_n) \leftrightarrow \varphi^t(v_1, ..., v_n))$. \end{itemize}
There is also a seemingly weaker version of RP, which asserts that any true statement is already true in some transitive set containing the parameters.
\begin{itemize}
    \item [] (RP$^-$) $\forall v_1, ... v_n [\varphi (v_1, ..., v_n) \rightarrow \exists t (\{v_1, ... v_n\} \subseteq t \ \land \ t \text{ is transitive } \land \ \varphi^t(v_1, ..., v_n))].$
\end{itemize} 
L\'evy and Vaught \cite{levy1961principles} showed that over Zermelo set theory, RP$^-$ does not imply RP.

\subsection{Dependent choice scheme} 

The Dependent Choice scheme (studied in \cite{Gitman2016-GITWIT} and \cite{Friedman2019-FRIAMO-8}), as a class version of the Axiom of Dependent Choice (DC), asserts that if $\varphi$ defines a class relation without terminal nodes, then there is an infinite sequence threading this relation.
\begin{itemize}
    \item [] (DC-scheme) If for every $x$ there is some $y$ such that $\varphi(x, y, u)$, then for every $p$ there is an infinite sequence $s$ such that $s(0) = p$ and $\varphi(s(n), s(n+1), u)$ for every $n<\omega$.
\end{itemize}

\noindent DC can be generalized to DC$_\kappa$ for any infinite well-ordered cardinal $\kappa$ as follows (first introduced by L\'evy in \cite{bwmeta1.element.bwnjournal-article-fmv54i1p13bwm}). 
\begin{itemize}
    \item [] (DC$_\kappa$) For every $x$ and $r\subseteq x^{<\kappa} \times x$, if for every $s \in x^{<\kappa}$, there is some $w \in x$ such that $\<s, w> \in r$, then there is an $f: \kappa \rightarrow x$ such that $\<f\restriction\alpha, f(\alpha)> \in r$ for all $\alpha < \kappa$.
\end{itemize}
\noindent Similarly, we can formulate the class version of DC$_\kappa$ for any $\kappa$.
\begin{itemize}
        \item[]  (DC$_\kappa$-scheme) If for every $x$ there is some $y$ such that $\varphi(x, y, u)$, then there is some function $f$ on $\kappa$ such that $\varphi(f\restriction \alpha, f(\alpha), u)$ for every $\alpha <\kappa$.
\end{itemize}
We say that DC$_{<Ord}$ holds if the DC$_\kappa$-scheme holds for every $\kappa$. One can verify that the DC$_\omega$-scheme is indeed a reformulation of the DC-scheme; moreover, the DC$_\kappa$-scheme is equivalent to the following (\cite[Proposition 13]{yao2023set}).
\begin{itemize}
\item [] For every definable class $X$, if for every $s \in X^{<\kappa}$ there is some $y \in X$ with $\varphi(x, y, u)$, then there is some function $f : \kappa \rightarrow X$ such that $\varphi(f\restriction \alpha, f(\alpha), u)$ for every $\alpha < \kappa$.
\end{itemize}
It is observed in \cite[page 397]{Gitman2016-GITWIT} that over ZF without Powerset, Collection and the DC$_\omega$-scheme jointly imply RP. The same argument goes through in $\ZFUR$ as well.
\begin{theorem}\label{collection+dc->rp}
$\ZFUR \vdash $ Collection $\land$ DC$_\omega$-scheme $\rightarrow$ RP. \qed
\end{theorem}

\subsection{Urelement axioms and homogeneity}
\begin{definition}
Let $A$ be a set of urelements.
\begin{enumerate}
   \item A set $x$ is \textit{realized by} $A$ if it is equinumerous with $A$ (abbreviated by $x \sim \A$); a set is \textit{realized} if it is realized by some set of urelements.  
    \item A set of urelements $B$ \textit{duplicates} $A$, if $B$ and $A$ are disjoint and $A \sim B$.
    \item A set of urelements $B$ is a \textit{tail} of $A$, if $B$ is disjoint from $A$ and for every $C \subseteq \A$ disjoint from $A$ there is an injection from $C$ to $B$.
\end{enumerate}
\end{definition}
\noindent Since $\A - A$ is a tail of $A$ whenever $\A$ is a set, the notion of tail can be seen as a generalized notion of \textit{complement} in terms of equinumerosity. When AC holds, the \textit{tail cardinal} of $A$ is the cardinality of a tail of $A$ if it exists. We then have the following axioms concerning urelements.
\begin{enumerate}
    \item [] (Plenitude) Every ordinal is realized.
    \item [] (Closure) The supremum of a set of realized ordinals is realized.
    \item [] (Duplication) Every set of urelements has a duplicate.
    \item [] (Tail) Every set of urelements has a tail.
\end{enumerate}

A key notion that will be frequently used is \textit{homogeneity over} $A$, which is originally introduced in \cite{hamkins2022reflection}.
\begin{definition}
\textit{Homogeneity holds over} a set of urelements $A$, if whenever $B \cup C \subseteq \A$ is disjoint from $A$ and $B \sim C$, there is an automorphism $\pi$ such that $\pi B = C$ and $\pi$ point-wise fixes $A$.
\end{definition}
\noindent Intuitively, when homogeneity holds over $A$, all equinumerous sets of urelements outside $A$ are indistinguishable from the perspective of $A$.

\begin{lemma}[$\ZFCUR$]\label{homogeneitylemma}
\
\begin{enumerate}

\item Homogeneity holds over some $A \subseteq \A$.

\item If $A \subseteq A' \subseteq \A$ and homogeneity holds over $A$, then homogeneity holds over $A'$.

\item Every $A \subseteq \A$ is a subset of some $A'$ over which homogeneity holds.
\end{enumerate}
\end{lemma}
\begin{proof}
 To show that homogeneity holds over some $A \subseteq \A$, it suffices to show that every infinite $D \subseteq \A$ disjoint from $A$ has a duplicate $D'$ that is also disjoint from $A$. To see this, suppose that $D = B \cup C$ is disjoint from $A$ and $B \sim C$. Then fix some duplicate $D'$ of $D$ that is disjoint from $A$, within which there is a duplicate $B'$ of $B$. We can then let $\pi_1$ be an automorphism that swaps $B$ and $B'$ while point-wise fixing $A$ and $\pi_2$ be an automorphism that swaps $B'$ and $C$ while point-wise fixing $A$. The composition of $\pi_1$ and $\pi_2$ is then a desired automorphism.

(1) Suppose Tail holds. Let $A$ be a set of urelements with the least tail cardinal $\kappa$. Suppose that $D$ is disjoint from $A$, which has size at most $\kappa$. Let $E$ be a tail of $D \cup A$. By the minimality of $\kappa$, $E$ must have size $\kappa$; so for some $E'\subseteq E$, $E' \sim D$. $E'$ is then a duplicate of $D$ that is disjoint from $A$. If Tail does not hold, then $\A$ is a proper class and we can fix an $A \subseteq \A$ without tails. Suppose that $D \subseteq \A$ is disjoint from $A$. We may assume $D$ is infinite and hence has size $\kappa$ for some infinite cardinal $\kappa$. Then there must be some $E \subseteq \A$ of size $\kappa^+$ that is disjoint from $A$, which means $E - D$ has size $\kappa^+$. So $E$ has a subset which duplicates $D$.

(2) Suppose that $A \subseteq A' \subseteq \A$ and homogeneity holds over $A$. Let $D$ be an infinite set of urelements disjoint from $A'$. Let $E = (A'- A) \cup D$. Since $E$ is infinite, by AC it can be partitioned into a pair of duplicates $E_1$ and $E_2$ such that $E \sim E_1 \sim E_2$. By homogeneity over $A$, there is an automorphism $\pi$ such that $\pi E_1 = E$ and $\pi$ point-wise fixes $A$. Since $E_1$ has a duplicate disjoint from $A$, by applying $\pi$ it follows that $E$ has a duplicate disjoint from $A$. Let $E'$ be such a duplicate of $E$, which is also disjoint from $A'$. Since $D \subseteq E$, $E'$ has a subset that duplicates $D$. Hence, homogeneity holds over $A'$.

(3) is an immediate consequence of (1) and (2).
\end{proof}
\noindent AC cannot be dropped in Lemma \ref{homogeneitylemma}: there are models of $\ZFUR$ where homogeneity holds over no set of urelements (\cite[Theorem 46]{yao2023set}).

\subsection{Implication diagram}
\begin{theorem}\label{maintheorem1}
Over $\ZFCUR$, the following implication diagram holds. The diagram is complete: if the diagram does not indicate that $\varphi$ implies $\psi$, then $\ZFCUR$ $+$ $\varphi \nvdash \psi$ if $\ZFCUR$ is consistent.
\end{theorem}
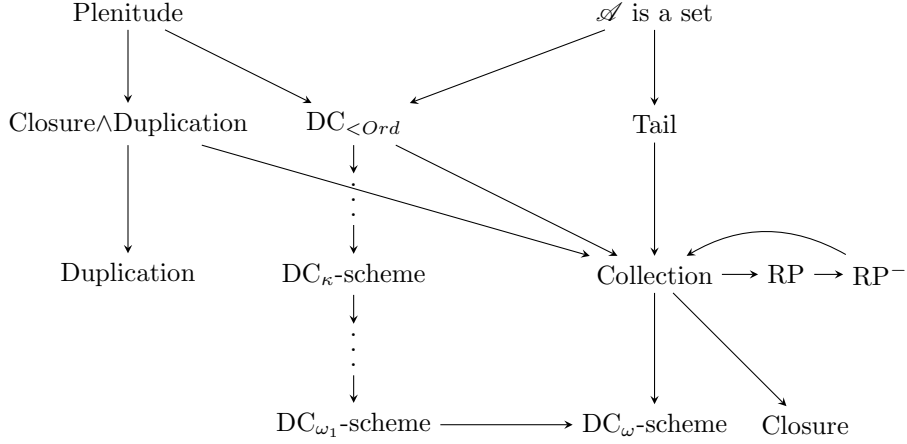
\begin{figure}[hbt!]
\begin{center}
\begin{tikzpicture}
\begin{scope}[every node/.style={}]
    \node (A) at (10, -2.5) {DC$_\omega$-scheme};
    \node (B) at (10, 1.5){Tail};
    \node (C) at (3,3) {Plenitude};
     \node (D) at (6,1.5) { DC$_{<Ord}$ };
    \node (E) at (6, -0.5) {DC$_\kappa$-scheme};
       \node (F) at (12, -2.5) {Closure};
          \node (G) at (11.75, -0.5) {RP};

    \node (H) at (3, 1.5) {\textup{Closure$\land$Duplication}};
    \node (I) at (10, -0.5) {Collection};
        \node (J) at (13, -0.5) {RP$^-$};

    \node (L) at (6, 0.7) {.};
    \node (M) at (6, 0.5) {.};
    \node (N) at (6, 0.3) {.};
    \node (O) at (3, -0.5) {Duplication};
    \node (P) at (6, -1.7) {.};
    \node (Q) at (6, -1.5) {.};
    \node (R) at (6, -1.3) {.};
    \node (S) at (6, -2.5) {DC$_{\omega_1}$-scheme};
    \node (T) at (10, 3) {$\A$ is a set};
    
\end{scope}

\begin{scope}[>={stealth},
              every node/.style={fill=white,circle},
              every edge/.style={draw=black}]
 
    \path [->] (T) edge (B);
    \path [->] (T) edge (D);
    \path [->] (C) edge (H);
    \path [->] (B) edge (I);
    \path [->] (I) edge (A);

     \path [->] (S) edge (A);

    \path [->] (H) edge (I);
    \path [->] (H) edge (O);
    
     \path [->] (C) edge (D);

    \path [->] (D) edge (I);
    \path [->] (D) edge (L);
    \path [->] (N) edge (E);
    \path [->] (E) edge (R);
     \path [->] (P) edge (S);
    
    \path [->] (I) edge (F);

    \path [->] (I) edge (G);
   \path [->] (G) edge (J);
   
       \path [->] (J) edge[bend right=30] (I); 
\end{scope}
\end{tikzpicture}
 \caption{Implication Diagram in $\ZFCUR$}
 \label{ZFCUdiagram}
\end{center}
\end{figure}
\FloatBarrier
\noindent The fact that Collection implies DC$_\omega$-scheme is first proved by Schlutzenberg in an answer to a question on Mathoverflow \cite{387471}; Plenitutde and Tail are implicitly discussed in Schlutzenberg's proof. All other non-trivial implications in Diagram \ref{ZFCUdiagram} are new, many of which, such as Tail $\rightarrow$ Collection, will be useful for the later discussions of forcing. The proof of Collection $\rightarrow$ DC$_\omega$-scheme in this paper also takes a different route by using Lemma \ref{homogeneitylemma}. The rest of this subsection establishes all the non-trivial implications in the diagram; the completeness of the diagram will be proved in \ref{section2.5}.

We first prove that Plenitude implies DC$_{<Ord}$. Given a formula $\varphi(x, y, u)$ with some parameter $u$, for any ordinals $\alpha, \alpha', \kappa, \kappa'$ and set of urelements $E$, we say that $\<\kappa' \alpha'>$ is a $(\varphi, E)$\textit{-extension} of $\<\kappa, \alpha>$ if (i) $\alpha \leq \alpha'$, and (ii) whenever $A \subseteq \A$ extends $E$ by $\kappa$-many urelements, there is some $B \subseteq \A$ disjoint from $A$ with $B \sim \kappa'$ such that for every $x \in V_{\alpha}(A)$, there is some $y \in V_{\alpha'} (A \cup B)$ such that $\varphi (x, y, u)$.

\begin{lemma}[$\ZFCUR$]\label{phiextension}
Suppose that Plenitude holds and $\forall x \exists y \varphi(x, y, u)$. Then every $\langle \kappa, \alpha \rangle$ has a $(\varphi, ker(u))$-extension.
\end{lemma}
\begin{proof}
By the remark in the proof of Lemma \ref{homogeneitylemma}, Plenitude implies that homogeneity holds over every set of urelements. Fix some $\langle \kappa, \alpha \rangle$ and some $A \subseteq \A$ extending $ker(u)$ with $\kappa$-many urelements. For each $x\in V_\alpha(A)$, define $\theta_x$ to be the least cardinal such that there is some $y$ with $\varphi(x, y, u)$ and $ker(y) \sim \theta_x$, and let $\kappa' = Sup\{ \theta_x : x \in  V_\alpha(A)\}$. Fix some infinite $B$ of size $\kappa'$ that is disjoint from A, which exists by Plenitude. Then for every $ x \in V_\alpha(A)$, fix some $y'$ such that $\varphi(x, y', u)$ and $ker(y') \sim \theta_x$. $ker(y') - A$ is equinumerous to a subset of $B$, so by homogeneity over $A$, there is an automorphism $\pi$ that moves $ker(y')$ into $B$ and point-wise fixes $A$. It follows that $\varphi(x, \pi y', u)$ and $\pi y' \in V(A \cup B)$. Thus, each $x \in V_\alpha(A)$ has some $y \in V(A\cup B)$ with $\varphi(x, y, u)$, so there is some large enough $\alpha'$ such that every $x \in V_\alpha(A)$ has some $y \in V_{\alpha'}(A\cup B)$ with $\varphi(x, y, u)$. Furthermore, for every $A'$ extending $ker(u)$ by $\kappa$-many urelements, by homogeneity over $ker(u)$, there is an automorphism $\pi$ with $\pi A = A'$ that point-wise fixes $\ker(u)$; so $\pi B$ will be such that every $x \in V_\alpha(A')$ has some $y \in V_{\alpha'}(A' \cup \pi B)$ with $\varphi(x, y, u)$. Therefore, $\<\kappa', \alpha'>$ is indeed a $(\varphi, ker(u))$-extension of $\<\kappa, \alpha>$.
\end{proof}

\begin{theorem}[$\ZFCUR$]\label{Plenitude->DCS}
Plenitude $\rightarrow$ DC$_{<Ord}$.
\end{theorem}
\begin{proof}
Suppose that Plenitude  holds. It suffices to show that the DC$_\kappa$-scheme holds for every regular cardinal $\kappa$ since by a standard argument as in \cite[Theorem 8.1]{jech2008axiom}, this will imply the DC$_\lambda$-scheme for every singular $\lambda$ as well. So suppose that $\kappa$ is regular and $\forall x \exists y \varphi(x, y, u)$ with some parameter $u$. We will construct a set $\barx$ that is closed under $<$$\kappa$-sequences (i.e., $\barx^{<\kappa} \subseteq \barx$) and the $\varphi$-relation (i.e., $\forall x \in \barx \exists y \in \barx \varphi(x, y, u)$). This will suffice for the DC$_\kappa$-scheme because we can apply DC$_\kappa$ to $\barx$ to get a desired $\kappa$-sequence.

We first construct a $\kappa$-sequence of pairs of ordinals $\langle \langle \lambda_\alpha, \gamma_\alpha \rangle : \alpha < \kappa \rangle$ by recursion as follows. Let $A_0$ be a set of urelements that extends $\ker(u)$ by $\lambda_0$-many urelements and $\gamma_0$ be an ordinal with cf$(\gamma_0) = \kappa$. For each ordinal $\alpha < \kappa$, we let $\langle \lambda_{\alpha+1} , \gamma_{\alpha+1} \rangle$ be the lexicographical-least $(\varphi, ker(u))$-extension of $\<\lambda_\alpha ,\gamma_\alpha>$ with cf$(\gamma_\alpha) = \kappa$, which exists by Lemma \ref{phiextension}; we take the union at the limit stage. By Plenitude, there exists a $\kappa$-sequence of sets of urelements $\langle A_\alpha : \alpha < \kappa \rangle$, where $A_\alpha$ extends $\bigcup_{\beta < \alpha} A_\beta \cup ker(u)$ by $\lambda_\alpha$-many urelements. Let $\barx = \bigcup_{\alpha < \kappa} V_{\gamma_\alpha} (A_\alpha)$. For any $x \in V_{\gamma_\alpha} (A_\alpha)$, there is some $B$ disjoint from $A_\alpha$ witnessing the fact that $\<\lambda_{\alpha +1}, \gamma_{\alpha+1}>$ is a $(\varphi, ker(u))$-extension of $\<\lambda_\alpha ,\gamma_\alpha>$. By homogeneity over $A_\alpha$, it follows that $A_{\alpha +1} - A_\alpha$ is a witness as well, so there is some $y \in V_{\gamma_{\alpha +1}}(A_{\alpha +1})$ with $\varphi(x, y, u)$. This shows that $\barx$ is closed under the $\varphi$-relation. Finally, $\barx^{<\kappa} \subseteq \barx$ because $\textup{cf} (\gamma_\alpha) = \kappa$ for each $\alpha < \kappa$ and $\kappa$ is regular. This completes the proof.
\end{proof}

\begin{theorem}[$\ZFCUR$]\label{Plenitude->Collection}
\
\begin{enumerate}
\item Closure $\land$ Duplication $\rightarrow$ Collection.
\item Plenitude $\rightarrow$ (Closure $\land$ Duplication $\land$ Collection).
\end{enumerate}
\end{theorem} 
\begin{proof}
(1) Assume  Closure and Duplication. Suppose that $\forall x \in w \exists y \varphi (x, y, u)$, where $w$ is a set and $u$ is a parameter. For every $x \in w$, let $\theta_x$ be the least $\theta$ realized by the kernel of some $y$ such that $\varphi(x, y, u)$, and define $\theta$ as the supremum of all such $\theta_x$. Let $A \subseteq \A$ be such that $ker(w) \cup \ker(u) \subseteq A$ and homogeneity holds over $A$,  which exists by Lemma \ref{homogeneitylemma} (3). By Closure and Duplication, there is a $B \subseteq \A$ of size $\theta$ that is disjoint from $A$. Then for every $x \in w$, fix a $y'$ such that $\varphi(x, y', u)$ and $ker(y') \sim \theta_x$. By homogeneity over $A$, there is an autormophism that moves $ker(y')$ into $A \cup B$ while point-wise fixing $A$. Thus, every $x \in w$ has a $y \in V(A \cup B)$ such that $\varphi(x, y, u)$. Therefore, Collection holds by applying Proposition \ref{weakcollection}.

(2) Assume Plenitude. Closure immediately follows. By (1), it remains to show Duplication holds. Given an infinite $A \subseteq \A$, $A \sim \kappa$ for some infinite cardinal $\kappa$ and we can fix some $B$ that realizes $\kappa^+$. So $B$ contains a subset that duplicates $A$.
\end{proof}

\begin{theorem}[$\ZFCUR$]\label{tail->collection}
Tail $\rightarrow$ Collection
\end{theorem}
\begin{proof}
Assume that every set of urelements has a tail and suppose that $\forall x \in w \exists y \varphi (x, y, u)$. There is an $A \subseteq \A$ such that $ker(w) \cup ker(u) \subseteq A$ and homogeneity holds over $A$. Let $B$ be a tail of $A$. For every $x \in w$ and $y'$ such that $\varphi(x, y', u)$, $ker(y') - A$ can be mapped injectively into $B$. So by homogeneity over $A$, there is an automorphism that moves $ker(y')$ into $A \cup B$ while point-wise fixing $A$. Therefore, every $x \in w$ has some $y \in V(A \cup B)$ such that $\varphi (x, y, u)$ and hence Collection holds by Proposition \ref{weakcollection}.
\end{proof}

The next two lemmas demonstrate how the DC$_\kappa$-scheme is related to urelements.
\begin{lemma}[$\ZFCUR$]\label{lemma:DCK->KRealized}
Let $\kappa$ be an infinite cardinal. Assume that the DC$_\kappa$-scheme holds and $\A$ is a proper class. Then $\kappa$ is realized.
\end{lemma}
\begin{proof}
Since $\A$ is a proper class, for every $x$, there is some $y$ with $ker(x) \subsetneq ker(y)$. It follows from the DC$_\kappa$-scheme that there exists a $\kappa$-sequence $f$ with $ker(f\restriction \alpha) \subsetneq ker(f(\alpha))$ for every $\alpha < \kappa$. Fix a well-ordering $\prec$ of $ker(f)$. Then the map $$\alpha \mapsto \text{the}\prec\text{-least element of }ker(f(\alpha +1)) - ker(f\restriction(\alpha + 1))$$ is an injection from $\kappa$ to $ker(f)$. Therefore, $\kappa$ is realized.
\end{proof}

\begin{lemma}[$\ZFCUR$]\label{Tailkappa->DCkappa}
Let $\kappa$ be an infinite cardinal and suppose that every set of urelements has a tail of size at least $\kappa$. Then the DC$_\kappa$-scheme holds.
\end{lemma}
\begin{proof}
Again, we may assume $\kappa$ is regular. Suppose that $\forall x \exists y \varphi(x, y, u)$ with some parameter $u$. Let $A$ be a set of urelements extending $ker(u)$ over which homogeneity holds and $B$ be a tail of $A$. Since $B$ has size at least $\kappa$, $B$ can be partitioned into $\kappa$-many pieces $\{B_\eta : \eta < \kappa\}$, where $B_\eta \sim B$ for each $\eta$. We define a $\kappa$-sequence of ordinals $\<\gamma_\alpha : \alpha < \kappa>$ by recursion, where $\gamma_\alpha$ is the least ordinal such that 
\begin{itemize}
    \item [] (i) $\gamma_\alpha > \bigcup_{\eta < \alpha} \gamma_\eta$ and cf$(\gamma_\alpha) = \kappa$;
    \item [] (ii) for every $x$ in $ \bigcup_{\eta < \alpha} V_{\gamma_\eta}(\bigcup_{\eta < \alpha}B_\eta \cup A)$, there is a $y \in V_{\gamma_\alpha}(\bigcup_{\eta \leq \alpha} B_\eta \cup A)$ with $\varphi(x, y, u)$.
\end{itemize}
Such $\gamma_\alpha$ exists, because homogeneity holds over $\bigcup_{\eta < \alpha}B_\eta \cup A$ by Lemma \ref{homogeneitylemma} (2) and each $B_\eta$ is a tail of $A$, which allows us to find a sufficiently large $\gamma_\alpha$. Let $\barx = \bigcup_{\alpha < \kappa} V_{\gamma_\alpha}(\bigcup_{\eta \leq \alpha}B_\eta \cup A)$. It follows that $\barx$ is closed under the $\varphi$-relation and $\barx^{<\kappa} \subseteq \barx$.
We can then apply DC$_\kappa$ to $\barx$ to get a desired $\kappa$-sequence.\end{proof}

\begin{theorem}[$\ZFCUR$]\label{easyimplication}
\
\begin{enumerate}
    \item $\A$ is a set $\rightarrow$ DC$_{<Ord}$. 
    \item DC$_{<Ord}$ $\rightarrow$ Collection.
    \item RP$^-$ $\rightarrow$ Collection.
    \item Collection $\rightarrow$ Closure.
    \item Collection $\rightarrow$ (Plenitude $\lor$ Tail).
    \item Collection $\rightarrow$ DC$_\omega$-scheme.
    \item Collection $\rightarrow$ RP.
\end{enumerate}
Hence, the implication diagram \ref{ZFCUdiagram} holds over $\ZFCUR$.
\end{theorem}
\begin{proof}

(1) Assume $\A$ is a set and $\forall x \exists y \varphi (x, y, u)$. Fix some regular $\kappa$. Define $\langle \gamma_\alpha : \alpha < \kappa \rangle$ by recursion by letting $\gamma_\alpha$ be the least ordinal such that cf$(\alpha) = \kappa$ and $\forall x \in \bigcup_{\eta <  \alpha} V_{\gamma_\eta} (\A) \exists y \in V_{\gamma_\alpha}(\A) \varphi (x, y, u)$. Let $\bar{x} = \bigcup_{\alpha < \kappa} V_{\gamma_\alpha}(\A)$. As before, we can then apply DC$_\kappa$ to $\barx$ to get a desired $\kappa$-sequence.

(2) By Lemma \ref{lemma:DCK->KRealized}, DC$_{<Ord}$ implies that either Plenitude holds or $\A$ is a set. But (Plenitude $\lor$ $\A$ is a set) implies Collection by Proposition \ref{weakcollection} and Theorem \ref{Plenitude->Collection}, so DC$_{<Ord}$ implies Collection.

(3) Suppose that RP$^-$ holds. Again, by Theorem \ref{Plenitude->Collection} we may assume that Plenitude fails. Now we show that Tail holds, which suffices by Theorem \ref{tail->collection}. So fix an $A \subseteq \A$ and let $\kappa$ be the least cardinal that is not realized by any set of urelements disjoint from $A$. $\kappa$ exists because Plenitude fails. Then by RP$^-$ there is a transitive set $t$ extending $\{\kappa, A\}$ such that $t$ reflects $\forall \lambda < \kappa \exists B (B \sim \lambda \land B \cap A = \emptyset)$. It is not hard to check $C = \bigcup \{B \in t : B \subseteq \A \land B \cap A = \emptyset\}$ is a tail of $A$.

Now assume Collection. 

(4) Let $x$ be a set of realized cardinals. Then there is a set $y$ such that for every $\kappa \in x$, there is some $A \in y$ such that $A \sim \kappa$. Let $B = \bigcup \{A: A \in y\}$. Then the cardinality of $B$ is at least the supremum of $x$ and hence Closure holds. 

(5) Suppose that Plenitude fails. Given a set $A$ of urelements, let $w$ be the set of cardinals realized by some $B \subseteq \A$ disjoint from $A$. By Collection there is some set $v$ such that for every $\lambda \in w$, there is some $B \in v$ such that $B \sim \lambda$ and $B \cap A = \emptyset$. $C = \bigcup \{B \in v: B \subseteq \A \land B \cap A = \emptyset\}$ is a tail of $A$. 

(6) To show the DC$_{\omega}$-scheme holds, we may assume that Plenitude fails by Theorem \ref{Plenitude->DCS} and that $\A$ is a proper class by (1). Then by Collection and (5), every set of urelements must have an infinite tail, so the DC$_{\omega}$-scheme holds by Lemma \ref{Tailkappa->DCkappa}.

(7) RP holds by (6) and Theorem \ref{collection+dc->rp}.
\end{proof}
\noindent As a consequence, Collection and Reflection can be characterized in terms of urelements in $\ZFCUR$.
\begin{corollary}[$\ZFCUR$]\label{urcharacterization}
The following are equivalent.
\begin{enumerate}
    \item RP
    \item RP$^-$
    \item Collection
    \item Plenitude $\lor$ Tail
\end{enumerate}
\end{corollary}
\begin{proof}
(1) $\to$ (2) is immediate. (2) $\to$ (3) and (3) $\to$ (4) are proved in Theorem \ref{easyimplication}. (4) $\rightarrow$ (1) by Theorem \ref{Plenitude->Collection} and \ref{tail->collection} and \ref{easyimplication} (7). \end{proof}
Many of the implications in Diagram \ref{ZFCUdiagram} no longer hold without AC. For example, $\ZFUR + $ Plenitude cannot prove either Duplication or Collection, as shown in \cite[Theorem 36]{yao2023set}. Meanwhile, many questions regarding $\ZFUR$ are open.
\begin{question}
Does $\ZFUR$ prove any of the following?
\begin{enumerate}
    \item Tail $\rightarrow$ Collection.
    \item Collection $\rightarrow$ RP$^-$.
    \item RP$^- \rightarrow$ Collection.
\end{enumerate}
\end{question}

\subsection{Completeness of the diagram}\label{section2.5}
We shall prove the completeness of Diagram \ref{ZFCUdiagram} by an easy method of building inner models of $\ZFCUR$, which is implicitly used in \cite{levy1969definability} and \cite{felgner1976choice}.

\begin{definition}\label{normalideal}
A \textit{class} $\I$ of sets of urelements is an $\A$-\textit{ideal} if 
\begin{enumerate}
    \item $\A \notin \I$ (if $\A$ is a set);
    \item if $A, B \in \I$, then $A \cup B \in \I$; 
    \item if $A \in \I$ and $B \subseteq A$, then $B \in \I$;
    \item for every $a \in \A$, $\{a\} \in \I$.
\end{enumerate}
Given an $\A$-ideal $\I$, $U^\I = \{x \in U : ker(x) \in \I\}$, i.e., the class of objects whose kernel is in $\I$.
\end{definition}
 
 As in ZF, classes are always understood as definable classes. When $X$ is a class and $\pi$ is a permutation of $\A$, $\pi X$ denotes the class $\{\pi x \mid x \in X\}$.

\begin{lemma}\label{idealpermutation}
Let $\I$ be an $\A$-ideal. Then for every $a, A$ such that $a \in A \in \I$, there is a permutation $\pi$ of $\A$ such that (i) $\pi \I = \I$, (ii) $\pi a \neq a$ and (iii) $\pi$ point-wise fixes $A - \{a\}$.
\end{lemma}
\begin{proof}
Fix some $a' \in \A - A$. Let $\pi$ be a permutation that only swaps $a$ and $a'$. To see that $\pi \I = \I$, let $B \in \I$. We may assume $a \in B$ and $a' \notin B$. Then $\pi B = (B - \{a\}) \cup \{a'\}$, which is in $\I$. Also, $B = \pi ((B - \{a\}) \cup \{a'\})$. Therefore, $\pi \I = \I$.
\end{proof}

\begin{theorem}[$\ZFCUR$]\label{smallkernelmodel}
Let $\I$ be an $\A$-ideal. Then $U^\I \models \ZFCUR$ + ``$\A$ is a proper class''.
\end{theorem}
\begin{proof}
It is clear that $U^\I$ is transitive and contains all the urelements and pure sets. Thus, $U^\I$ satisfies Foundation, Extensionality, Infinity, and $\A$ is a proper class in  $U^\I$. It is also immediate that $U^\I$ satisfies Pairing, Union, Powerset and Separation. $U^\I \models$ AC because for any set $x$ in $U^\I$, any bijection in $U$ between $x$ and an ordinal has the same kernel as $x$ and hence also lives in $U^\I$. It remains to show that Replacement holds in $U^\I$.

Suppose that $U^\I \models \forall x \in w \exists ! y \varphi (x, y, u)$ for some $w, u \in U^\I$. Let $v = \{ y \in U^\I : \exists x \in w\ \varphi^{U^\I}(x, y, u) \}$, which is a set in $U$. It suffices to show that $ker(v) \subseteq  ker(w) \cup ker(u)$. Suppose not. Then there are some $y$ and $a$ such that $y \in v$, $a \in ker(y)$ and $a \notin ker(w) \cup ker(u)$. Let $A = ker(w) \cup ker(u) \cup ker(y)$, which is in $\I$. By Lemma \ref{idealpermutation}, there is an automorphism $\pi$ such that (i) $\pi \I = \I$, (ii) $\pi a \neq a$ and (iii) $\pi$ point-wise fixes $A - \{a\}$. So $\pi$ point-wise fixes $w$ and $u$. Since $y \in v$, there is some $x \in w$ with $\varphi^{U^\I}(x, y, u)$. It follows that $\varphi^{U^\I}(x, \pi y, u)$, but $\pi y \neq y$ because $\pi a \in ker(\pi y)$ but $\pi a \notin ker(y)$, which contradicts the uniqueness of $y$.
\end{proof}

\begin{theorem}\label{zfcurindependece}
Assume the consistency of $\ZFCUR$. Over $\ZFCUR$,
\begin{enumerate}
     \item (Closure $\land$ Duplication) $\nrightarrow$ (Plenitude $\lor$ DC$_{\omega_1}$-scheme);
    \item  Collection $\nrightarrow$ Duplication;
    \item  Duplication $\nrightarrow$ (Closure $\lor$ DC$_\omega$-scheme);
    \item  Closure $\nrightarrow$ DC$_\omega$-scheme;
    \item  DC$_\kappa$-scheme $\nrightarrow$ Closure, where $\kappa$ is any infinite cardinal;
    \item (Collection $\land$ DC$_\kappa$-scheme) $\nrightarrow$ DC$_\lambda$-scheme, where $\kappa < \lambda$ are infinite cardinals.
\end{enumerate}
Hence, Diagram \ref{ZFCUdiagram} is complete.
\end{theorem}
\begin{proof}
In each case, $U$ is a model of $\ZFCUR$. These models exist if we assume the consistency of ZF (see \cite[Theorem 10]{yao2023set}), which follows from the consistency of $\ZFCUR$.

(1) Assume that in $U$, $\A \sim \omega_1$. Let $\I_1$ be the ideal of all countable subsets of $\A$. It is clear that in $U^{\I_1}$, Closure but Duplication hold. Plenitude fails in $U^{\I_1}$ because $\omega_1$ is not realized; since $\A$ is a proper class in $U^{\I_1}$, the DC$_{\omega_1}$-scheme fails by Lemma \ref{lemma:DCK->KRealized}. 

(2) Assume that in $U$, $\A \sim \omega_1$. Fix an $A \subseteq \A$ such that $A \sim \omega_1$ and $\A - A \sim \omega_1$. Let $\I_2 = \{ B \subseteq \A:  B  - A \text{ is countable} \}$. For every $B \in U^{\I_2}$, there will be a countable $C$ that is disjoint from $A\cup B$. Then $C \cup (A- B)$ is a tail of $B$, so Collection holds in $U^{\I_2}$ by Theorem \ref{tail->collection}. Duplication fails because $A$ has no duplicates in $U^{\I_2}$.

(3) Assume that in $U$, $\A \sim \omega$. Let $\I_3$ be the ideal of finite subsets on $\A$. It is cleat that in $U^{\I_3}$ Duplication holds and Closure fails. The DC$_\omega$-scheme fails in $U^{\I_3}$ by Lemma \ref{lemma:DCK->KRealized}.

(4) Assume that in $U$, $\A \sim \omega$ and fix an infinite and co-infinite $A \subseteq \A$. Let $\I_4 = \{ B \subseteq \A: B - A \text{ is finite}\}$. Closure holds in $U^{\I_4}$ because $\omega$ is the greatest realized cardinal. The DC$_\omega$-scheme fails in $U^{\I_4}$ since every set of urelements can be properly extended by another set of urelements disjoint from $A$, yet there cannot be a corresponding infinite sequence.

(5) Let $\kappa$ be an infinite cardinal. Assume that in $U$, $\A \sim \omega_{\kappa^+}$. Let $\I_5$  be the set of sets of urelements of size less than $\omega_{\kappa^+}$. Closure fails in $U^{\I_5}$ because $ \omega_{\kappa^+}$ is not realized while every cardinal below $\omega_{\kappa^+}$ is. To show that the DC$_\kappa$-scheme holds, suppose that for every $x \in U^{\I_5}$, there is some $y \in U^{\I_5}$ such that $\varphi^{U^{\I_5}}(x, y, u)$. Since the DC$_\kappa$-scheme holds in $U$ by Theorem \ref{easyimplication} (1) and $U^{\I_5}$ is closed under $\kappa$-sequences, in $U$ there is a function $f: \kappa \rightarrow U^{\I_5}$ such that $\varphi^{U^{\I_5}}(f\restriction\alpha, f(\alpha), u)$ for every $\alpha < \kappa$, which also exists in $U^{\I_5}$.

(6) It suffices to show that for any infinite cardinal $\kappa$, $\ZFCUR$ + Collection + DC$_\kappa$-scheme does not prove the DC$_{\kappa^+}$-scheme. Assume that in $U$, $\A \sim \kappa^+$ and let $\I_6$ be the ideal of all sets of urelements of size less than $\kappa^+$. $\kappa^+$ is not realized so by Lemma \ref{lemma:DCK->KRealized}, the DC$_{\kappa^+}$-scheme fails in $U^{\I_6}$. Every set of urelements in $U^{\I_6}$ has tail cardinal $\kappa$, so Collection holds by Theorem \ref{tail->collection} and the DC$_\kappa$-scheme holds by Lemma \ref{Tailkappa->DCkappa}.
\end{proof}

\subsection{What is ZFC with urelements?}
There is a close analogy between \textit{ZFC with urelements} and \textit{ZFC without Powerset}. For instance, Zarach \cite{zarach1996replacement} showed that ZFC without Powerset formulated with only Replacement (now commonly denoted by ZFC-) cannot prove Collection.
Moreover, it is shown in \cite{Gitman2016-GITWIT} that ZFC- has various pathological models, e.g., there are models of ZFC- where the \L o\'s theorem fails for some internal ultrapowers (\cite[Theorem 2.15]{Gitman2016-GITWIT}). All of these pathological models, however, can be excluded by Collection, and for this reason it is argued in \cite{Gitman2016-GITWIT} that ZFC without Powerset \textit{should} be formulated with Collection. 

This analogy can be strengthened if we consider internal ultrapowers in $\ZFCUR$. Let $U$ be a model of $\ZFCUR$ and $F, x \in U$ be such that $U \models (F$ is an ultrafilter on $x)$. For every $f, g \in U$ such that $U \models $ ($f, g$ are functions on $x$), define
\begin{itemize}
    \item [] $f =_F g \text{ if and only if } U \models (\{y \in x: f(y) = g(y) \} \in F);$
    \item [] $[f] = \{h \in U : (h \text{ is a function on } x)^U \land h =_F f\};$
    \item [] $U/F = \{ [h] : h \in U \land (h \text{ is a function on } x)^U\}$.
\end{itemize}
For every $[f], [g] \in U/F$, define
\begin{itemize}
    \item [] $[g] \hat{\in} [f] \text{ if and only if } U \models (\{y \in x : g(y) \in f(y)\} \in F);$
    \item [] $\hat{\A}([f]) \text{ if and only if } U \models (\{y \in x : \A( f(y))\} \in F).$
\end{itemize}
Then the \textit{internal ultrapower} is the model $\<U/F, \ \hat{\in}, \ \hat{\A}>$ (denoted by $U/F$). The \L o\'s theorem holds for $U/F$ if for every $\varphi$ and $[f_1], ..., [f_n] \in  U/F$,
\begin{itemize}
    \item [] $U/F \models \varphi ([f_1], ..., [f_n])$ if and only if $\{y \in x : \varphi(f_1(y), ..., f_n(y))\} \in F.$
\end{itemize}
When $V \models $ ZFC, the \L o\'s theorem holds for all internal ultrapowers of $V$.

\begin{theorem}\label{thm:collection<->losthm}
Let $U$ be a model of $\ZFCUR$. The following are equivalent.
\begin{enumerate}
    \item $U \models$ Collection.
    \item The \L o\'s theorem holds for all internal ultrapowers of $U$.
\end{enumerate}
\end{theorem}
\begin{proof}
The proof of (1) $\rightarrow$ (2) is standard, and the point here is that the use of Collection is essential. 

(2)$\rightarrow$(1). Suppose that Collection fails in $U$. By Diagram \ref{ZFCUdiagram}, it follows that both Plenitude and Tail fail in $U$. In $U$, fix some $A \subseteq \A$ that does not have a tail and let $\kappa$ be the least cardinal that is not realized by any $B\subseteq \A$ disjoint from $A$, which is an infinite limit cardinal. Let $F \in U$ be an ultrafilter on $\kappa$ containing all the unbounded subsets of $\kappa$. Suppose \textit{for reductio} that the \L o\'s theorem  holds for $U/F$. Let $id$ be the identity function on $\kappa$ and $c_A$ be the constant function sending every $\alpha < \kappa$ to $A$. Since $\{ \alpha < \kappa : \exists B \subseteq \A \ (B \sim \alpha \land B \cap A =\emptyset)\} \in F$, by  the \L o\'s theorem $U/F \models \exists B \subseteq \A (B \sim [id] \land B \cap [C_A] =\emptyset)$. Thus, there is some $g \in U$ such that 
$$U/F \models [g] \subseteq \A \land [g] \sim [id] \land ([g] \cap [C_A] =\emptyset).$$
Let $x= \{\alpha < \kappa : g(\alpha) \subseteq \A \land g(\alpha) \sim \alpha \land (g(\alpha) \cap A = \emptyset) \}$, which is in $F$ by the \L o\'s theorem again. Then $D = \bigcup_{\alpha\in x} g(\alpha)$ has size $\kappa$ and is disjoint from $A$---contradiction.
\end{proof}
Later we will see that over $\ZFCUR$, Collection is also equivalent to the principle that every (properly defined) forcing relation is full (Theorem \ref{collection<->fullness}). These results suggest that Collection should be part of a robust axiomatization of ZFC with urelements.

\section{Forcing with urelements}\label{section3}
\subsection{Existing approach}
We now turn to forcing over countable transitive models of $\ZFCUR$. Given a forcing poset $\P$, a natural thought is that each urelement behaves as a different copy of $\emptyset$ and so we may treat every urelement as its own $\P$-name. This approach has been adopted in all existing studies such as \cite{blass1989freyd}, \cite{hall2002characterization} and \cite{Hall2007-ERIPMA}. 
\begin{definition}\label{oldpnames}
Let $\P$ be a forcing poset. $\dot{x}$ is a $\P$-name$_{\scaleto{\#}{4pt}}$ if and only if either $\dot{x}$ is a urelement, or $\dot{x}$ is a set of ordered-pairs $\langle \dot{y}, p \rangle$, where $\dot{y}$ is a $\P$-name$_{\scaleto{\#}{4pt}}$ and $p \in \P$. $U^{\P}_{\scaleto{\#}{4pt}} = \{ \dot{x} : \dotx \text{ is a } \P\text{-name}_{\scaleto{\#}{4pt}}\}$.
\end{definition}

\begin{definition}\label{oldmg}
Let $M$ be a countable transitive model of $\ZFUR$, $\P \in M$ be a forcing poset, and $G$ be an $M$-generic filter over.
\begin{enumerate}
    \item $M^\P_{\scaleto{\#}{4pt}} = M \cap U^\P_{\scaleto{\#}{4pt}}$.
    \item For every $\dotx \in M^\P_{\scaleto{\#}{4pt}}$, 
    \begin{equation*}
       \dotx_G =
    \begin{cases*}
       \dotx & if $\A(\dotx)$ \\
\{\dot{y}_G : \exists p \in G \langle \dot{y}, p \rangle \in \dot{x} \}     & otherwise 
    \end{cases*}
\end{equation*}

\item  $M[G]_{\scaleto{\#}{4pt}} = \{\dotx_G : \dotx \in M^\P_{\scaleto{\#}{4pt}}  \}$.
\item For every $\dotx_1, ..., \dotx_n \in  M^\P_{\scaleto{\#}{4pt}}$ and $p \in \P$, $p \forces_{\scaleto{\#}{4pt}} \varphi (\dot{x}_1, ..., \dot{x}_n)$ if and only if for every $M$-generic $H$ such that $p \in H$, $M[H] \models \varphi (\dotx_{1_H}, ..., \dotx_{n_H})$.
\end{enumerate}
\end{definition}
\noindent With these definitions, one can easily prove the forcing theorems for $\forces_{\scaleto{\#}{4pt}}$ by making trivial adjustments to the standard argument. And it is clear that $M[G]_{\scaleto{\#}{4pt}}$ is transitive, $M \subseteq M[G]_{\scaleto{\#}{4pt}}$, and $G \in M[G]_{\scaleto{\#}{4pt}}$. In fact, $M[G]_{\scaleto{\#}{4pt}}$ is a countable transitive model of $\ZFUR$ (see Appendix).

However, an important feature of forcing is missing in this approach, which is why the subscript $\#$ is added. Given $M$ and $\P$ as above, the forcing relation $\forces_{\scaleto{\#}{4pt}}$ given by $\P$ is said to be \textit{full} if whenever $p \forces_{\scaleto{\#}{4pt}} \exists y \varphi(y, \dotx_1, ..., \dotx_n)$ for $\dotx_1, ..., \dotx_n \in M^\P_{\scaleto{\#}{4pt}}$, there is a $\doty \in M^\P_{\scaleto{\#}{4pt}}$ such that $p \forces_{\scaleto{\#}{4pt}} \varphi(\doty, \dotx_1, ..., \dotx_n)$. It is a standard result that if $M \models$ ZFC, then for every forcing poset in $M$, its forcing relation is full. Fullness is important for various forcing constructions such as iterated forcing and Boolean-valued ultrapowers.

\begin{remark}\label{oldnamenotfull}
Let $M$ be a countable transitive model of $\ZFUR$ with urelements. Then for every $\P$ with a maximal antichain with at least two elements, its forcing relation $\forces_{\scaleto{\#}{4pt}}$ is not full.
\end{remark}
\begin{proof}
 Suppose that $\P \in M$ has a maximal antichain $\<p_i : i \in I>$ indexed by some $I$ ($|I| > 1$). Let $\<a_i : i \in I>$ be some urelements such that at least two of them are distinct. Consider the $\P$-name$_{\scaleto{\#}{4pt}}$ $\dotx = \{\<a_i, p_i> : i \in I \}$. It follows that $1_\P \forces_{\scaleto{\#}{4pt}} \exists y (y \in \dotx)$. But if $1_\P \forces_{\scaleto{\#}{4pt}} \doty \in \dotx$ for some $\doty \in M^\P_{\scaleto{\#}{4pt}}$, then $\doty$ must be some $a_i$, which is impossible since one can take an $M$-generic filter containing $p_j$ for some $j \neq i$.
\end{proof}
\noindent A diagnosis is that $M^\P_{\scaleto{\#}{4pt}}$ contains too few names. In pure set theory, whenever $f$ is a function from an antichain in a forcing poset $\P$ to some $\P$-names, we can define a \textit{mixture} of $f$, $\doty$, such that $p \forces f(p) = \doty$ for every $p \in dom(f)$. But as we have seen, $M^\P_{\scaleto{\#}{4pt}}$ does not even contain a mixture of two urelements. This motivates a new definition of $\P$-names with urelements.

\subsection{New approach}
\begin{definition}\label{newpnames}
Let $\P$ be a forcing poset. $\dot{x}$ is a $\P$-name if and only if (i) $\dot{x}$ is a set of ordered-pairs $\langle y, p \rangle$ where $p\in \P$ and $y$ is either a $\P$-name or a urelement, and (ii) whenever $\langle a, p \rangle, \langle y, q\rangle \in \dot{x}$, where $a$ is a urelement and $a \neq y$, $p$ and $q$ are incompatible. $U^\P = \{ \dotx : \dotx \text{ is a } \P \text{-name}\}$.
\end{definition}
\begin{definition}\label{newmg}
Let $M$ be a countable transitive model of $\ZFUR$, $\P \in M$ be a forcing poset and $G$ be an $M$-generic filter over $\P$.
\begin{enumerate}
    \item $M^\P = U^\P \cap M$
   \item For every $\dot{x} \in M^\P$, 
    \begin{equation*}
    \dot{x}_G =
    \begin{cases*}
       a & if $\A(a)$ and  $\langle a, p \rangle \in \dot{x}$ for some $p \in G$   \\
 \{ \dot{y}_G: \langle \doty , p \rangle \in \dotx \text{ for some } \doty \in M^\P \text{ and } p \in G \}      & otherwise 
    \end{cases*}
\end{equation*}
   
    \item $M[G] = \{\dotx_G : \dot{x} \in M^\P \}$. 
    
    \item For every urelement $a \in M$, $\check{a} = \{\<a, 1_\P>\}$; for every set $x \in M$, $\check{x} = \{\<\check{y}, 1_\P> : y \in x\}$.
\end{enumerate}
For every $\dotx_1, ..., \dotx_n \in  M^\P$ and $p \in \P$, $p \forces \varphi (\dot{x}_1, ..., \dot{x}_n)$ if and only if for every $M$-generic $G$ such that $p \in G$, $M[G] \models \varphi (\dotx_{1_G}, ..., \dotx_{n_G})$.
\end{definition}

 Let us explain the idea behind this new forcing machinery. First, no urelement is a $\P$-name in $U^\P$, and each urelement $a$ is represented by $\{\<a, 1_\P >\}$ rather than itself. Second, when $\<a, p> \in \dotx$ for some urelement $a$, this indicates that $a$ will be \textit{identical to}, rather than \textit{a member of}, $\dotx_G$ for any generic filter $G$ containing $p$. This motivates the incompatibility condition (ii) in Definition \ref{newpnames} by the following reasoning. Suppose that $a$ is a urelement and $\langle a, p \rangle \in \dot{x}$. If $ \langle b, q\rangle \in \dot{x}$, where $b$ is a different urelement, then $p$ must be incompatbile with $q$ since $\dotx$ must not be interpreted as two different urelements in any generic extension. Also, if  $ \langle \doty, q\rangle \in \dotx$, where $\doty$ is a $\P$-name, then $p$ must be incompatbile with $q$ as well since otherwise $\doty$ would become a member of $\dotx$ in some generic extension where $\dotx$ is interpreted as a urelement. We note that the two forcing methods we have seen produce the same forcing extensions (see Appendix).

\begin{lemma}\label{McoversM[G]}
Let $M$ be a countable transitive model of $\ZFUR$, $\P \in M$ be a forcing poset, and $G$ be an $M$-generic filter over $\P$. Then
\begin{enumerate}
    \item  $M \subseteq M[G]$.
    \item  $G \in M[G]$.
    \item  $M[G]$ is transitive.
    \item  $Ord \cap M = Ord \cap M[G]$.
    \item  For every transitive model $N$ of $\ZFUR$ such that $G\in N$ and $M \subseteq N$, $M[G] \subseteq N$.
    \item $\A \cap M = \A \cap M[G]$.
    \item $ker(\dot{x}_G) \subseteq ker(\dot{x})$ for every $\dot{x} \in M^{\P}$.
    \item For every set of urelements $A \in M[G]$, there is a set of urelements $A' \in M$ such that $A \subseteq A'$.
    \item $M \models$ ($\A$ is a set) if and only if $M [G] \models$ ($\A$ is a set).
\end{enumerate}
\end{lemma}
\begin{proof}
(1)--(5) are all proved by standard text-book arguments as in \cite[Ch.VII]{kunen2014set}. 

(6) This is clear by the construction of $M[G]$ because every urelement in $M[G]$ must come from $ker(\dotx)$ for some $\dotx \in M^\P$. 

(7) This is proved by induction on the rank of $\dot{x}$. We may assume that $\dotx_G$ is a set. Since $ker(\dot{x}_{G}) \subseteq \bigcup \{ ker(\dot{y}_G) : \dot{y} \in dom(\dot{x})\}$ and by the induction hypothesis we have $ker(\dot{y}_G) \subseteq ker(\dot{y}) \subseteq ker(\dot{x})$ for every $\dot{y} \in dom(\dot{x})$, it follows that $ker(\dot{x}_G) \subseteq ker(\dot{x})$.

(8) For every $A \in M[G]$, $A = \dotx_G$ for some $\dotx \in M^\P$ so by (7) we have $A = ker(\dotx_G) \subseteq ker(\dotx)$.

(9) The left-to-right direction follows from (4) and (6). For the other direction, let $A$ be the set of all urelements in $M[G]$. By (8), $A \subseteq A'$ for some $A' \in M$. So if $a$ is an urelement in $M$, it follows from (6) that $a \in A'$.
\end{proof}

Next we need to prove the forcing theorems for $\forces$, i.e., ``$p \forces \varphi$'' is definable in $M$ for every $\varphi$, and every truth in a generic extension is forced by some condition in the corresponding generic filter. The first step is to define an internal forcing relation.

\begin{definition}
Let $M$ and $\P$ be as before. The forcing language $\mathcal{L}^M_{\P}$ contains $\{\subseteq, =, \in, \A, \overset{\mathscr{A}}{=}\}$ as the non-logical symbols and every $\P$-name in $M^\P$ as a constant symbol. For every $p \in \P$ and $\varphi \in \mathcal{L}^M_{\P}$, we define $p \forces^* \varphi$ by recursion as follows.
\begin{enumerate}
    \item $p \forces^* \A(\dotx_1)$ if and only if $\{q \in \P : \exists \<a , r> \in \dotx_1 \ (\A(a) \land q \leq r)\}$ is dense below $p$.
    \item $p \forces^* \dotx_1 \Aeq \dotx_2$ if and only if $\{q \in \P : \exists a, r_1, r_2 (\A(a) \land \<a, r_1> \in \dotx_1 \land \<a, r_2> \in \dotx_2 \land q \leq r_1, r_2 )\} \cup \{q \in \P : \forall \<a_1, r_1> \in \dotx_1 \ (\A(a_1) \rightarrow q \bot r_1) \land \forall \<a_2, r_2> \in \dotx_2 \ (\A(a_2) \rightarrow q \bot r_2)\}$ is dense below $p$. 
    \item $p \forces^* \dotx_1 \in \dotx_2$ if and only if $\{ q \in \P : \exists \<\doty, r> \in \dotx_2 (q \leq r \land \doty \in M^\P \land q \forces^* \doty = \dotx_1)\}$ is dense below $p$.
    \item $p \forces^* \dotx_1 \subseteq \dotx_2$ if and only if for every $\doty \in M^\P$ and $r, q \in \P$, if $\<\doty, r> \in \dotx_1$ and $q \leq p, r$, then $q \forces^* \doty \in \dotx_2$.
    \item $p \forces^* \dotx_1 = \dotx_2$ if and only if $p \forces^* \dotx_1 \subseteq \dotx_2$, $p \forces^* \dotx_2 \subseteq \dotx_1$ and $p \forces^* \dotx_1 \Aeq \dotx_2$.
    \item $p \forces^* \neg \varphi$ if and only if there is no $q \leq p$ such that $q \forces^* \varphi$.
    \item $p \forces^* \varphi \land \psi$ if and only if $p \forces^* \varphi$ and $p \forces^* \psi$.
    \item $p \forces^* \exists x \varphi$ if and only if $\{q \in \P : \text{ there is some } \dotz \in M^{\P} \text{ such that }q \forces^* \varphi(\dotz)\}$ is dense below $p$.
\end{enumerate}
\end{definition}
\noindent The idea behind the predicate $\Aeq$ is that if $p \forces^* \dotx_1 \Aeq \dotx_2$, then in every generic filter $G$ containing $p$, either $\dotx_{1_G}$ and $\dotx_{2_G}$ are the same urelement, or neither of them is a urelement.
\begin{lemma}\label{forcingbasic}
Let $M$ and $\P$ be as before. For every $p, q \in \P$,
\begin{enumerate}
    \item If $p \forces^* \varphi$ and $q \leq p$, then $q \forces^* \varphi$.
    \item If $\{r \in \P : r \forces^* \varphi \}$ is dense below $p$, $p \forces^* \varphi$. \qed
\end{enumerate}
\end{lemma}

\begin{lemma}\label{forcinglemma}
Let $M$ be a countable transitive model of $\ZFUR$, $\P \in M$ be a forcing poset and $G$ be an $M$-generic filter over $\P$. For every $\dotx_1, ..., \dotx_n \in M^\P$,
\begin{enumerate}
    \item If $p\in G$ and $p \forces^* \varphi(\dotx_1, ..., \dotx_n)$, then $M[G] \models \varphi (\dotx_{1_G}, ..., \dotx_{n_G})$.
    \item If $M[G] \models \varphi (\dotx_{1_G}, ..., \dotx_{n_G})$, then there is some $p \in G$ such that $p \forces^* \varphi(\dotx_1, ..., \dotx_n)$.
\end{enumerate}
\end{lemma}
\begin{proof}
Since the Boolean and quantifier cases can be proved in the same way as in \cite[Chapter VII. Theorem 3.5]{kunen2014set}, we omit their proofs. It remains to show that the lemma holds for all atomic formulas, which we shall prove by induction on the rank of the $\P$-names.\\

\noindent \textit{Case 1}. $\varphi$ is $\dotx_1 \in \dotx_2$. The argument is the same as in \cite[Chapter VII, Theorem 3.5]{kunen2014set}.\\

\noindent \textit{Case 2}. $\varphi$ is $\A(\dotx)$.  For (2), suppose that $\dotx_G$ is some urelement $b$. Then $\<b, p> \in \dotx$ for some $p \in G$, so $\{q \in \P : \exists \<a , r> \in \dotx \ (\A(a) \land q \leq r)\}$ is dense below $p$ and hence $p \forces^* \A(\dotx)$. For (1), suppose that $p \forces^* \A(\dotx)$ for some $p \in G$. Then there is some $q \in G$ such that $\<b, r> \in \dotx$ for some $r \geq q$ and urelement $b$. Thus, $\dotx_G = b$.\\

\noindent \textit{Case 3}. $\varphi$ is $\dotx_1 = \dotx_2$. For (2), suppose that $\dotx_{1_G} = \dotx_{2_G}$.\\

\textit{Subcase 3.1}. $\dotx_{1_G} = \dotx_{2_G} = b$ for some urelement $b$. Then $\<b, s_1> \in \dotx_1$ and $\<b, s_2> \in \dotx_2$ for some $s_1, s_2 \in G$. Fix some $p \in G$ such that $p \leq s_1, s_2$. Observe first that $p \forces^* \dotx_1 \subseteq \dotx_2$ and $p \forces^* \dotx_2 \subseteq \dotx_1$ trivially hold: for any $\P$-name $\doty$ and $r \in \P$ such that $\<\doty, r> \in \dotx_1 (\text{or } \dotx_2)$, $p$ must be incompatible with $r$ because $r$ is incompatible with $s_1(\text{or }s_2)$. Moreover, $p \forces^* \dotx_1 \Aeq \dotx_2$ because $\{q \in \P : \exists a, r_1, r_2 (\A(a) \land \<a, r_1> \in \dotx_1 \land \<a, r_2> \in \dotx_2 \land  q \leq r_1, r_2 )\}$ is clearly dense below $p$. Hence, $p \forces^* \dotx_1 = \dotx_2$.\\

\textit{Subcase 3.2}. $\dotx_{1_G}$ is a set. We first use a standard text-book argument to show that $p \forces^* \dotx_1 \subseteq \dotx_2$ and $p \forces^* \dotx_2 \subseteq \dotx_1$ for some $p \in G$. Define:
\begin{itemize}
    \item [] $D_1= \{p \in \P : p \forces^* \dotx_1 \subseteq \dotx_2 \land p \forces^* \dotx_2 \subseteq \dotx_1 \}$;
    \item [] $D_2 = \{p \in \P : \exists \<\doty_1, q_1> \in \dotx_1 \ (p \leq q_1 \land \forall \<\doty_2, q_2> \in \dotx_2 \ \forall r \leq q_2 \ (r \forces^* \doty_1 = \doty_2 \rightarrow p \bot r )) \}$;
    \item [] $D_3 = \{p \in \P : \exists \<\doty_2, q_2> \in \dotx_2 \ (p \leq q_2 \land \forall \<\doty_1, q_1> \in \dotx_1 \ \forall r \leq q_1 \  (r \forces^* \doty_2 = \doty_1 \rightarrow p \bot r )) \}$.
\end{itemize}
 If $p \nVdash^* \dotx_1 \subseteq \dotx_2$, then there are $\<\doty_1, q_1> \in \dotx_1$ and $r \leq p, q_1$ such that $r \nforces^* \doty_1 \in \dotx_2$; so there is an $s \leq r$ such that for every $\<\doty_2, q_2> \in \dotx_2$ and $s' \leq q_2$. If $s' \forces^* \doty_1 = \doty_2$, then $s \bot s'$. Hence, $s \leq p$ and $s \in D_2$. Similarly, if $p \nforces^* \dotx_2 \subseteq \dotx_1$, then $p$ will have an extension in $D_3$. This shows that $D_1 \cup D_2 \cup D_3$ is dense. However, $G \cup (D_2 \cup D_3)$ must be empty. Suppose \textit{for reductio} that $p \in G\cap D_2$. Fix some $\<\doty_1, q_1> \in \dotx_1$ with $p \leq q_1$ that witnesses $p \in D_2$. It follows that $\doty_1{_G} = \doty_2{_G}$ for some $\<\doty_2, q_2> \in \dotx_2$ with $q_2 \in G$. By the induction hypothesis, there is some $r \in G$ such that $r \leq q_2$ and $r \forces^* \doty_1 = \doty_2$. But $p$ must be incompatible with such $r$, which is a contradiction. The same argument shows that $G \cap D_3$ is empty. Therefore, there is some $p \in G$ such that $p \forces^* \dotx_1 \subseteq \dotx_2$ and $p \forces^* \dotx_2 \subseteq \dotx_1$.

Now we wish to find some $q \in G$ such that $q \forces^* \dotx_1 \Aeq \dotx_2$. Define:
\begin{itemize}
    \item [] $E_1 = \{ q \in \P : \forall r \leq q \ [\forall \<a_1, s_1> \in \dotx_1 \ (\A(a) \rightarrow r \bot s_1) \land \forall \<a_2, s_2> \in \dotx_2 \ (\A (a_2) \rightarrow r \bot s_2)]\}$;
    \item [] $E_2 = \{ q \in \P : \exists \<a, r> \in \dotx_1 \ (\A(a) \land q \leq r) ) \}$;
    \item [] $E_3 = \{ q \in \P : \exists \<a, r> \in \dotx_2 \ (\A(a) \land q \leq r)\}$.
\end{itemize}
$E_1 \cup E_2 \cup E_3$ is dense. But if there is some $q \in G \cap (E_2 \cup E_3)$, either $\dotx_{1_G}$ or $\dotx_{2_G}$ would be a urelement. Thus there is some $q \in G \cap E_1$ such that the set
\begin{itemize}
    \item [] $\{r \in \P : \forall \<a_1, s_1> \in \dotx_1 \ (\A(a_1) \rightarrow r \bot s_1) \land \forall \<a_2, s_2> \in \dotx_2 \ (\A (a_2) \rightarrow r \bot s_2)\}$
\end{itemize}
is dense below $q$. Therefore, $q \forces^* \dotx_1 \Aeq \dotx_2$. A common extension of $p$ and $q$ in $G$ will then force $\dotx_1 = \dotx_2$.\\

To show that (1) holds for Case 3, suppose that for some $p \in G$, $p \forces^* \dotx_1 = \dotx_2$.\\

\textit{Subcase 3.3}. $\dotx_{1_G} = b$ for some urelement $b$. Then $\<b, r> \in \dotx_1$ for some $r \in G$. Define:
\begin{itemize}
    \item [] $F_1 = \{ q \in \P : \exists a, s_1, s_2 (\A (a) \land \<a, s_1> \in \dotx_1 \land \<a, s_2> \in \dotx_2 \land q \leq s_1, s_2) \}$.
    \item [] $F_2 = \{q \in \P : \forall \<a, s_1> \in \dotx_1 \ (\A(a) \rightarrow q \bot s_1) \land \forall \<a, s_2> \in \dotx_2 \ (\A(a) \rightarrow q \bot s_2 )\}$.
\end{itemize}
Since $p \forces^* \dotx_1 \Aeq \dotx_2$, $F_1 \cup F_2$ is dense below $p$. But clearly $F_2 \cap G$ is empty as $\<b, r> \in \dotx_1$, so there is some $q \in F_1 \cap G$. It follows that $\<b, s_1> \in \dotx_1$ and $\<b, s_2> \in \dotx_2$ for some $s_1, s_2 \in G$. Therefore, $\dotx_{2_G} = b = \dotx_{2_G}$.\\

\textit{Subcase 3.4}. $\dotx_{1_G}$ is a set. Suppose \textit{for reductio} that $\dotx_{2_G}$ is some urelement $b$ and so $\<b, r> \in \dotx_2$ for some $r \in G$. Since $p \forces^* \dotx_1 \Aeq \dotx_2$, it follows that there are some urelement $a$ and  $s \in G$ such that $\<a, s> \in \dotx_1$. This implies that $\dotx_{1_G} = a$, which is a contradiction. Hence, $\dotx_{2_G}$ is a set, so it remains to show that $\dotx_{1_G}$ and $\dotx_{2_G}$ have the same members. If $\doty_G \in \dotx_{1_G}$, then $\<\doty, r> \in \dotx_1$ for some $r \in G$. So there is some $q \in G$ with $q \leq p, r$, and $q \forces^* \doty \in \dotx_2$ because $p \forces^* \dotx_1 \subseteq \dotx_2$. By the induction hypothesis, $\doty_G \in \dotx_{2_G}$. The same argument will show that $\dotx_{2_G} \subseteq \dotx_{1_G}$. \end{proof}
\begin{theorem}
Let $M$ be a countable transitive model of $\ZFUR$ and $\P \in M$ be a forcing poset. Then for every $\dotx_1, ..., \dotx_n \in M^\P$,
\begin{enumerate}
    \item $p \forces^* \varphi(\dotx_1, ..., \dotx_n)$ if and only if $p \forces \varphi(\dotx_1, ..., \dotx_n)$.
    \item For every $M$-generic filter $G$ over $\P$, $M[G] \models \varphi (\dotx_{1_G}, ..., \dotx_{n_G})$ if and only if $\exists p \in G (p \forces \varphi(\dotx_1, ..., \dotx_n))$. 
\end{enumerate}
\end{theorem}
\begin{proof}
    By a standard argument as in \cite[Chapter VII. Theorem 3.6]{kunen2014set} using Lemma \ref{forcinglemma}.
\end{proof}

\subsection{Fullness is equivalent to Collection.}
We first verify that $M^\P$ is closed under mixtures.
\begin{lemma}\label{mixlemma}
Let $M$ be a countable transitive model of $\ZFUR$ and $\P \in M$ be a forcing poset. For every function $f: dom(f) \rightarrow M^\P$ in $M$, where $dom(f)$ is an antichain in $\P$, there is a $\dot{v} \in M^\P$ such that $p \forces \dot{v} = f(p)$ for every $p \in dom(f)$.
\end{lemma}
\begin{proof}
In $M$, we define $\dot{v}$ as follows.
\begin{align*}
    \dotv= \bigcup_{p \in dom(f)}\{\<y, r> \in dom(f(p)) \times \P :  \exists q \  ( \<y, q> \in f(p) \land r \leq p, q )\}.
\end{align*}
We first check that $\dot{v}$ satisfies the incompatbility condition (ii) in Definition \ref{newpnames}. Consider any $\<a, r_1>\in \dot{v}$ for some urelement $a$. Then there are $p_1, q_1$ such that $p_1 \in dom(f)$ and  $\<a, q_1> \in f(p_1)$ and $r_1 \leq p_1, q_1$. For any $\<x, r_2> \in \dot{v}$ with $x \neq a$, there are $p_2, q_2$ such that $p_2 \in dom(f)$ and  $\<x, q_2> \in f(p_2)$ and $r_2 \leq p_2, q_2$. If $p_1 = p_2$, then $r_1$ is incompatible with $r_2$ because $f(p_1)$ is a $\P$-name. If not, $r_1$ is incompatible with $r_2$ because $dom(f)$ is an antichain.

Fix a $p \in dom(f)$. We show that $p \forces \dot{v} = f(p)$. Let $G$ be an $M$-generic filter over $\P$ that contains $p$.

\textit{Case 1}. $\dot{v}_{G}$ is some urelement $a$. Then $\<a, r> \in \dot{v}$ for some $r \in G$. So for some $p' \in dom(f)$ and $q$,  $\<a, q> \in f(p')$ and $r \leq p', q$. So $p', q \in G$ and $p' = p$. Therefore, $\dot{v}_{G} = f(p)_{G}$.

\textit{Case 2}. $\dot{v}_{G}$ is a set. Then $f(p)_{G}$ must be a set. Otherwise, $f(p)_{G}$ is some urelement $a$ and there will be some $q \in G$ such that $\<a, q> \in f(p)$; then there is some $s \in G$ such that $s \leq q, p$ so $\<a, s> \in \dot{v}$, which means $\dot{v}_{G}$ is a urelement----contradiction. For every $ \dot{x}_{G} \in \dot{v}_{G}$ with $\<\dot{x}, r> \in \dot{v}$ and $r \in G$, $\<\dot{x}, q> \in f(p')$ and $r \leq p', q$ for some $p' \in dom(f)$ and $q$; so $p'= p$ and $ \dot{x}_{G} \in f(p)_{G}$. This shows that $\dot{v}_{G} \subseteq f(p)_{G}$. Consider any $\dot{x}_{G} \in f(p)_{G}$ such that $\<\dot{x}, q> \in f(p)$ for some $q \in G$. Let $r \in G$ be a common extension of $p$ and $q$. It follows that $\<\dot{x}, r> \in \dot{v}$ and so $\dot{x}_{G} \in \dot{v}_{G}$. This shows that $f(p)_{G} \subseteq \dot{v}_{G}$ and the proof is completed.
\end{proof}

\begin{theorem}\label{collection<->fullness}
Let $M$ be a countable transitive model of $\ZFCUR$. The following are equivalent.
\begin{enumerate}
    \item  $M \models$ Collection.
    \item For every forcing notion $\P \in M$, its forcing relation $\forces$ is full.
\end{enumerate}
\end{theorem}
\begin{proof}
 (1) $\rightarrow$ (2). This is proved by a standard argument, and the point here is that we can now mix $\P$-names using the new definition. Assume that $M \models$ Collection and we now work in $M$. Fix some $\P$ and suppose that $p \forces \exists y \varphi(y)$ for some $p \in \P$. By AC there exists a maximal antichain $X$ in the subposet $\mathbb{Q} = \{ q \in \P : q \leq p \land \exists \dot{y} \in M^\P q \forces \varphi(\dot{y})\}$. It follows from Collection that there is some $v$ such that for every $q \in X$, there is some name $\dotw \in v$ with $q \forces \varphi(\dotw)$. By well-ordering $v$ we can pick a $\dot{w}_q \in M^\P$ such that $q \forces \varphi(\dot{w}_q)$ for every $q \in X$. Then by lemma \ref{mixlemma}, there is a $\dot{v} \in M^\P$ such that $q \forces \dot{w}_q  = \dot{v}$ for every $q \in X$. Suppose that $p \nVdash \varphi(\dot{v})$ \textit{for reductio}. Then there will be some $r \in \mathbb{Q}$ such that $r \forces \neg \varphi(\dot{v})$, which means $r$ is incompatible with every $q \in X$, but this contradicts the maximality of $X$.\\
 
 \noindent (2) $\rightarrow$ (1). Suppose that $M \models \forall x \in w \exists y \varphi(x, y, u)$. Let $\P$ be the forcing poset $w \cup \{w\}$ such that for every $p, q \in \P$, $p \leq q$ if and only if $p = q$ or $q = w$. That is, $w$ is $1_\P$, while the members of $w$ constitute the only maximal antichain. Thus, $M[G] = M$ for every generic filter $G$ over $\P$. Define $\dot{x} \in M^\P$ to be $\{ \<\check{z}, x> : z \in x \land x \in w\}$. For every $x \in w$ and generic filter $G$ containing $x$, since $\dot{x}_{G} = x$ it follows that $M[G] \models \exists y \varphi (\dot{x}_{G}, y, u)$. Thus, $1_\P \forces \exists y \varphi (\dot{x}, y, \check{u})$ and by (2), $1_\P \forces \varphi (\dot{x}, \doty, \check{u})$ for some $\doty \in M^\P$. For every $x \in w$, let $G$ be the generic filter containing $x$. Then $M[G] \models \varphi (x, \doty_{G}, u)$; so $M \models \varphi(x,  \doty_{G}, u)$ and $ker(\doty_{G}) \subseteq ker(\doty)$ by Lemma \ref{McoversM[G]} (7). It follows that $M \models \forall x \in w \ \exists y \in V(ker(\doty)) \ \varphi(x, y, u)$, which suffices for Collection by Proposition \ref{weakcollection}.\end{proof}
 \noindent While the argument for (2) $\rightarrow$ (1) in Theorem \ref{collection<->fullness} does not rely on the assumption that $M \models $ AC, the use of AC in (1) $\rightarrow$ (2) is essential (see \cite[Corollary 62.1]{yao2023set}).

\subsection{Axiom preservation}\label{section3.4}
\begin{lemma}\label{forcingpreservesZFCU}
Let $M$ be a countable transitive model of $\ZFUR$, $\P \in M$ be a forcing poset, and $G$ be an $M$-generic filter over $\P$. Then
\begin{enumerate}
    \item $M[G]$ is a countable transitive model of ZU (Definition \ref{def:ZUandZFU}).
    \item $M[G] \models $ AC if $M \models $ AC.
    \item $M[G] \models$ Collection if $M \models$ Collection.
\end{enumerate}
\end{lemma}
\begin{proof}
The proofs of (1) and (2) are standard text-book arguments as in Kunen \cite[Ch.VII]{kunen2014set}. For (3), suppose that $M[G] \models \forall v \in \dotw_G\ \exists y \varphi(v, y, \dot{u}_G)$ for some $\dotw_G$ and $\dot{u}_G$. In $M$, define 
\begin{align*}
    x = \{\langle \dotx, p \rangle \in (dom(\dotw) \cap M^\P) \times \P :  \exists \dot{y} \in M^\P p \forces \varphi(\dotx, \dot{y}, \dot{u}) \}.
\end{align*}
By Collection in $M$, there is a set of $\P$-names $v$ such that for every $\langle \dotx, p \rangle \in x$, there is a $\dot{y} \in v$ with $p \forces \varphi (\dotx, \dot{y}, \dot{u})$. Define $\dot{v}$ to be $v \times \{1_\P\}$. It's routine to check that $M[G] \models \forall x \in \dot{w}_G \ \exists y \in \dot{v}_G \ \varphi(x, y, \dot{u}_G)$.
\end{proof}
A more difficult question is whether forcing preserves Replacement. When $M$ is a model of ZF, the standard argument for $M[G] \models $ Replacement appeals to Collection in $M$, which is not available for us here. A new argument is thus needed.

\begin{definition}\label{purification}
Let $M$ and $\P$ be as before and $A \in M$ be a set of urelements. For every urelement $a \in M$, let $\overset{A}{a} = a$. For every $\dotx \in M^\P$, we define $\overset{A}{\dot{x}}$ (the $A$\textit{-purification of }$\dotx$) as follows.
\begin{align*}
   \overset{A}{\dot{x}} = \{\langle \overset{A}{y}, p \rangle : \langle y , p \rangle \in \dot{x} \land ( y \in M^\P \lor y \in A) \}. 
\end{align*}
\end{definition}
\noindent That is, $\overset{A}{\dot{x}}$ is obtained by hereditarily throwing out the urelements used to build $\dot{x}$ that are not in $A$.
\begin{prop}
Let $A \in M$ be a set of urelements such that $ker(\P) \subseteq A$. For every $\dotx \in M^\P$, $\overset{A}{\dot{x}} \in M^\P$ and $ker(\overset{A}{\dot{x}}) \subseteq A$.
\end{prop}
\begin{proof}
By induction on the rank of $\dotx$. To show that $\overset{A}{\dot{x}}$ is always a $\P$-name, we only need to check the incompatibility condition in Defnition \ref{newpnames} holds. Suppose that $\<a, p>, \<y, q> \in \overset{A}{\dot{x}}$, where $a$ is a urelement and $y \neq a$. If $y$ is another urelement in $dom(\dotx)$, then $p$ and $q$ are incompatible; otherwise $y$ is some $\overset{A}{\dotz}$, where $\<\dotz, q> \in \dotx$ and $\dotz$ is a $\P$-name, then $p$ and $q$ are incompatible because no urelement is a $\P$-name. $ker(\overset{A}{\dot{x}}) \subseteq A$ because $ker(\overset{A}{\dot{x}})$ is contained in $\bigcup_{y \in dom(\dotx)}ker(\overset{A}{y}) \cup ker(\P)$, which is a subset of $A$ by the induction hypothesis.
\end{proof}

\begin{theorem}\label{forcingpreservesreplacement}
Let $M$ be a countable transitive model of $\ZFUR$, $\P \in M$ be a forcing poset and $G$ be $M$-generic over $\P$. Then $M[G] \models$ Replacement.
\end{theorem}
\begin{proof}
Suppose that $M[G] \models \forall v \in \dot{w}_G \exists ! y \varphi(v, y, \dot{u}_G)$. Let $A = ker(\dot{w}) \cup ker(\P) \cup ker(\dot{u})$. By Lemma \ref{forcingpreservesZFCU}, we may assume $M$ does not satisfy Collection so it has a proper class of urelements by Proposition \ref{weakcollection}.

\begin{lemma}\label{keylemmarep}
For every $\dot{v}_G \in \dot{w}_G$, there exist $p \in G$ and $\mu' \in M^{\P}$ such that $p \forces \varphi(\dot{v}, \mu', \dot{u})$ and $ker(\mu') \subseteq A$.
\end{lemma}
\begin{proof}
Fix a $\dot{v}_G \in \dot{w}_G$ for some $\dot{v} \in dom (\dot{w}) \cap M^\P$. There is a $\P$-name $\mu$ and a $p \in G$ such that $ p \forces \varphi (\dot{v}, \mu, \dot{u}) \land \forall z (\varphi (\dot{v}, z, \dot{u}) \rightarrow \mu = z)$.
\begin{claim}\label{claim1}
 For every $M$-generic filter $H$ over $\P$ such that $p \in H$, $ker(\mu_H) \subseteq A$.
\end{claim}
\begin{claimproof}
Suppose not. Then there is some $b \in ker(\mu_H) - A$. Since $M$ has a proper class of urelements, there is some urelement $c \in M$ such that $c \notin A \cup ker(\mu)$. In $M$, let $\pi$ be an automorphism that only swaps $b$ and $c$. Since $\pi$ point-wise fixes $A$, it follows that 
\begin{align*}
    p \forces \varphi (\dot{v}, \pi \mu, \dot{u}) \land \forall z (\varphi (\dot{v}, z, \dot{u}) \rightarrow \pi \mu = z).
\end{align*}
Thus, $M[H] \models \mu_H = (\pi\mu)_H$. Since $b \in ker(\mu_H)$, $\pi b  \in ker(\pi\mu_H)$; but $\pi b = c \notin ker(\mu)$ and $ker(\mu_H) \subseteq ker(\mu)$, so $\pi b \notin ker(\mu_H)$, which is a contradiction.
\end{claimproof}

\noindent Note that we cannot hope to show that $ker(\mu) \subseteq A$ in general. For if $\mu^*$ is some $\P$-name such that  $\mu^* = \mu \cup \{\langle \{\<b, 1_\P>\}, q \rangle\}$, where $b$ is a urelement not in $A$ and $q$ is incompatible with $p$, we would still have $p \forces \mu = \mu^*$.

\begin{claim}\label{claim2}
Let $H$ be an $M$-generic filter over $\P$ such that $p \in H$. For every $\dot{x}, \dot{y} \in M^{\P}$, if $\dot{x}_H, \dot{y}_H \in TC(\{\mu_H\})$, then $\dot{x}_H = \dot{y}_H$ if and only if $(\overset{A}{\dot{x})}_H = (\overset{A}{\dot{y}})_H$.
\end{claim}
\begin{claimproof}
If $\dotx_H = \doty_H = a$ for some urelement $a$, then by Claim \ref{claim1} $a \in A$. It is easy to check that $(\overset{A}{\dot{y})}_H = (\overset{A}{\dot{x}})_H = a$. If $(\overset{A}{\dot{y})}_H = (\overset{A}{\dot{x}})_H = b$ for some urelement $b$, then $b \in A$ and it follows that $\dotx_H = \doty_H = b$.

So suppose $\dotx_H = \doty_H$ are sets in $TC(\{\mu_H\})$ and the claim holds for every $\dot{z} \in dom(\dot{x}) \cup dom(\dot{y})$. Clearly, $(\overset{A}{\dot{x})}_H$ and $(\overset{A}{\dot{y}})_H$ must also be sets. If $\overset{A}{\dot{z}}_H \in \overset{A}{\dot{x}}_H$ for some $\dotz \in M^\P \cap dom(\dotx)$, we have $\dot{z}_H \in \dot{y}_H = \dot{x}_H$. So there is some $\dot{w} \in M^\P \cap dom(\doty)$ such that $\dot{w}_H = \dot{z}_H$. $\dot{z}_H \in TC(\{\mu_H\})$ so by the induction hypothesis $\overset{A}{\dot{z}}_H = \overset{A}{\dot{w}}_H \in (\overset{A}{\dot{y}})_H$. This shows that $\overset{A}{\dot{x}}_H \subseteq \overset{A}{\dot{y}}_H$, and we will have $\overset{A}{\dot{x}}_H = \overset{A}{\dot{y}}_H$ by the same argument.

Now suppose that $\dotx_H, \doty_H \in TC(\{\mu_H\})$ and $\overset{A}{\dot{x}}_H = \overset{A}{\dot{y}}_H$ are sets. Then $\dotx_H$ and $\doty_H$ must be sets. For if, say, $\dotx_H = a$ for some urelement $a$, then $a \in A$ by Claim \ref{claim1}, which would yield $\overset{A}{\dot{x}}_H = a$. Let $\dot{z}_H \in \dot{x}_H$ for some $\dotz \in M^\P \cap dom(\dotx)$. Then $\overset{A}{\dot{z}}_H \in \overset{A}{\dot{y}}_H$ and so $\overset{A}{\dot{z}}_H = \overset{A}{\dot{w}}_H$ for some $\dot{w}_H \in \dot{y}_H$. By the induction hypothesis, it follows that $\dot{z}_H = \dot{w}_H$. This shows that $\dot{x}_H \subseteq \dot{y}_H$ and consequently, $\dot{x}_H = \dot{y}_H$.
\end{claimproof}

\begin{claim}\label{claim3}
$p \forces \overset{A}{\mu} = \mu$.
\end{claim}
\begin{claimproof}
 Let $H$ be an $M$-generic filter on $\P$ that contains $p$. We show that $\overset{A}{\mu}_H = \mu_H$. Let $f$ be the function on $TC(\{\mu_H\})$ that sends every $\dot{y}_H$ to $\overset{A}{\dot{y}}_H$, which is is well-defined by Claim \ref{claim2}. Note that every $\in$- isomorphism of transitive sets that fixes the urelements can only be the identity map. So it suffices to show that $f$ maps $TC(\{\mu_H\})$ onto $TC(\{\overset{A}{\mu}_H\})$, preserves $\in$ and fixes all the urelements.

\textit{ $f$ preserves $\in$}. Consider any $\dot{y}{_H}, \dot{x}{_H} \in TC(\{\mu_H\})$. Suppose that $\dot{y}{_H} \in \dot{x}{_H}$. Then $\dot{y}{_H} = \dot{z}_H$ for some $\dot{z} \in M^\P \cap dom(\dot{x})$ so $ \overset{A}{\dot{z}}_H \in \overset{A}{\dot{x}}_H$; by Claim \ref{claim2}, it follows that $\overset{A}{\dot{y}}_H = \overset{A}{\dot{z}}_H \in \overset{A}{\dot{x}}_H$. Suppose that $\overset{A}{\dot{y}}_H \in \overset{A}{\dot{x}}_H$. Then $\overset{A}{\dot{y}}_H = \overset{A}{\dot{z}}_H$ for some $\dot{z}_H \in \dot{x}_{H}$ so $\dot{y}{_H} = \dot{z}_H \in \dot{x}_{H}$ by Claim \ref{claim2} again.

\textit{$f$ maps $TC(\{\mu_H\})$ onto $TC(\{\overset{A}{\mu}_H\})$}. If $\dot{y}_H \in TC(\{\mu_H\})$, then $\dot{y}_H \in \dot{y}_{1_H} \in ... \in \dot{y}_{n_H} \in \mu_H$ for some $n$. Since $f$ is $\in$-preserving, it follows that $\overset{A}{\dot{y}}_H \in \overset{A}{\dot{y}_1}_H \in ... \in \overset{A}{\dot{y}_n}_H \in \overset{A}{\mu}_H$ and hence $\overset{A}{\dot{y}}_H \in TC(\{{\overset{A}{\mu} }_H\})$. To see it is onto, let $x \in x_1 \in ... \in x_n \in \overset{A}{\mu}_H$. Then $x = \overset{A}{\dot{y}}_H \in \overset{A}{\dot{y}_1}_H  \in ... \in \overset{A}{\dot{y}_n}_H \in \overset{A}{\mu}_H$, but then $\dot{y}_H \in \dot{y}_{1_H} \in ... \in \dot{y}_{n_H} \in \mu_H$ and hence  $\dot{y}_H \in TC(\{\mu_H\})$.

\textit{$f$ fixes all the urelements in  $TC(\{\mu_H\})$.} Suppose $\dotx_H = a \in TC(\{\mu_H\})$ for some urelement $a$. Then by Claim \ref{claim1}, $a \in A$ and hence $\overset{A}{\dot{x}}_H = a$. \end{claimproof}

\noindent Lemma \ref{keylemmarep} is now proved by letting $\mu'$ be $\overset{A}{\mu}$.
\end{proof}
Now in M, we define
\begin{align*}
    \bar{x} = \{\langle \dot{v}, p\rangle \in (dom(\dot{w}) \cap M^\P) \times \P : \exists \mu \in M^{\P} (ker(\mu) \subseteq A \land p \forces \varphi(\dot{v}, \mu, \dot{u})) \}.
\end{align*}
For every $\langle \dot{v}, p \rangle \in \bar{x}$, let $\alpha_{ \dot{v}, p }$ be the least $\alpha$ such that there is some $\mu \in V_\alpha(A) \cap M^{\P}$ such that $p \forces \varphi(\dot{v}, \mu, \dot{u})$. Let $\beta = Sup_{\langle \dot{v}, p \rangle \in \bar{x}} \alpha_{\dot{v}, p}$ and set $\rho = (V_\beta(A) \cap M^{\P}) \times \{1_\P\}$. It remains to show that $M[G] \models \forall x \in \dot{w}_G \ \exists y \in \rho_G \ \varphi(x, y, \dot{u}_G)$. Let $\dot{v}_G \in \dot{w}_G$. By Lemma \ref{keylemmarep}, there is some $p \in G$ such that $\langle \dot{v}, p \rangle \in \bar{x}$. So there is some  $\P$-name $\mu \in dom(\rho)$ such that $p \forces \varphi(\dot{v}, \mu, \dot{u})$. Thus, $M[G] \models \varphi(\dot{v}_G, \mu_G, \dot{u}_G)$ and $\mu_G \in \rho_G$.\end{proof}

\begin{theorem}\label{fundamentalthmofforcing}
Let $M$ be a countable transitive model of $\ZFUR$, $\P \in M$ be a forcing poset and $G$ be an $M$-generic fitler over $\P$. Then
\begin{enumerate}
    \item $M[G] \models$ $\ZFUR$.
    \item $M[G] \models$ AC if $M\models$ AC.
    \item $M[G] \models $ Collection if $M\models$ Collection.
    \item $M[G] \models$ Plenitude if $M \models $ Plenitude.
    \item $M[G] \models $ Duplication if $M \models $ Duplication.
    \item $M[G] \models$ Tail if $M \models$ Tail.
    \item $M[G] \models $ DC$_{<Ord}$ if $M \models $ DC$_{<Ord}$.
    \item $M[G] \models$ RP$^-$ if $M \models$ RP$^-$.
    \item $M[G] \models$ RP if $M \models$ RP.
    \item $M[G] \models $ Closure if $M \models $ Closure + AC.
\end{enumerate}
\end{theorem}
\begin{proof}
(1), (2) and (3) are proved in Lemma \ref{forcingpreservesZFCU} and Theorem \ref{forcingpreservesreplacement}. (4) is clear since if Plenitude holds in $M$, then every ordinal $\alpha$ in $M[G]$ is realized by some set of urelements in $M$. 

(5) Suppose that $M \models $ Duplication. Let $A \subseteq \A$ be in $M[G]$. By Lemma \ref{McoversM[G]} (8), $A \subseteq A'$ for some set $A'$ of urelements in $M$. $A'$ will have a duplicate in $M$, which has a subset that duplicates $A$ in $M[G]$. 

(6) Suppose that $M\models$ Tail. Let $A \subseteq \A$ be in $M[G]$. Fix some $A' \in M$ such that $A \subseteq A'$, and let $B'$ be a tail of $A'$ in $M$. We check that $D = (A'- A) \cup B'$ is a tail of $A$ in $M[G]$. Suppose that $C \in M[G]$ is disjoint from $A$. Fix $C' \in M$ with $C \subseteq C'$. Since $C' - A'$ injects into $B$, it follows that $C' - A$ and hence $C$ injects into $D$. Thus, $D$ is a tail of $A$. 

(7) Suppose that $M \models $ DC$_{<Ord}$. It is a standard result that $\forall \kappa \text{DC}_\kappa$ implies AC, so $M \models$ AC. Then by Lemma \ref{lemma:DCK->KRealized},  $(\A \text{ is a set} \lor \text{Plenitude})$ holds in $M$ and thus holds in $M[G]$ by (4). By (2), $M[G] \models$ AC so we can apply Theorem \ref{maintheorem1} to conclude that $M[G] \models$ DC$_{<Ord}$. 

(8) Suppose that $M \models$ RP$^-$ and $M[G] \models \varphi(\dotx_{1_G}, ..., \dotx_{n_G})$. Fix any $\dotu_G \in M[G]$. Let $p\in G$ be such that $p \forces \varphi(\dotx_1, ..., \dotx_n)$. By RP$^-$ in $M$, there is a transitive set $m$ extending $\{\P, \dotu, \dotx_1, ..., \dotx_n\}$ such that $(p \forces^* \varphi(\dotx_1, ..., \dotx_n))^m$ and $m$ satisfies some finite fragment of $\ZFUR$ that suffices for the construction of $\P$-names inside $m$ and for the forcing theorem to hold for $\varphi$. Then $m[G] \models \varphi(\dotx_{1_G}, ..., \dotx_{n_G})$. $m[G]$ is a set in $M[G]$ because $\dot{m}= \{\<\doty, 1_\P> : \doty \in m \cap M^\P\}$ is a $\P$-name for $m[G]$; $m[G]$ is transitive because $m$ is. Consequently, $M[G] \models \exists t (t \text{ transitive} \land \dotu_G \subseteq t \land \varphi^t(\dotx_{1_G}, ..., \dotx_{n_G}))$. This shows that RP$^-$ holds in $M[G]$.

(9) Suppose that $M \models$RP. Given a formula $\varphi (v_1, ..., v_n)$ and some $\dotu_G \in M[G]$, let $\psi (p, \P, v_1, ..., v_n)$ be the formula asserting that $p \forces^*\varphi(v_1, ..., v_n)$ for $\P$-names $v_1, ..., v_n$. By RP in $M$, there will a transitive set $m$ extending $\{\P, \dotu \}$ that fully reflects $\psi$ and satisfies some finite fragment of $\ZFUR$ sufficient for the construction of $\P$-names and for the forcing theorem to hold for $\varphi$. Then as in the last paragraph, $m[G]$ is a transitive set containing $\dotu_G$ in $M[G]$. If $M[G] \models \varphi (\dotx_{1_G}, ..., \dotx_{n_G})$ for some $\dotx_{1_G}, ..., \dotx_{n_G}$ in $m[G]$, then there will be a $p \in G$ such that $(p \forces^* \varphi(\dotx{_1}, ..., \dotx{_n}))^m$ by reflection, and so $m[G] \models \varphi (\dotx_{1_G}, ..., \dotx_{n_G})$. And if $M[G] \models \varphi (\dotx_{1_G}, ..., \dotx_{n_G})^{m[G]}$, then there is some $p \in G$ such that $(p\forces^* \varphi(\dotx{_1}, ..., \dotx{_n}))^m$, so $p\forces^* \varphi(\dotx{_1}, ..., \dotx{_n}) $ and hence $M[G] \models \varphi (\dotx_{1_G}, ..., \dotx_{n_G})$. This shows that $M[G] \models $ RP.

(10) Suppose that $M \models $ Closure$\land$AC. Let $X \in M[G]$ be a set of realized cardinals whose supermum is some limit cardinal $\lambda$. For every cardinal $\kappa < \lambda$ in $M$, there is a cardinal $\kappa'$ in $M[G]$ such that $\kappa < \kappa' < \lambda$ and $\kappa'$ is realized by some $A \in M[G]$; so in $M$ it follows from AC that any $A'$ that extends $A$ will have size at least $\kappa$. This shows that every $\kappa < \lambda$ in $M$ is realized, so $\lambda$ is realized in $M$ and hence in $M[G]$.
\end{proof}
\noindent  It is not known whether forcing over $\ZFUR$ preserves Closure.

\subsection{Axiom destruction and recovery}\label{section3.5}
We now turn to how forcing may destroy the DC$_\kappa$-scheme. It is known that forcing over ZF does not preserve DC$_\kappa$ for any $\kappa$ (\cite{Monro1983-MONOGE}), so we will focus on whether forcing preserves the DC$_\kappa$-scheme over $\ZFCUR$. A forcing poset $\P$ is \textit{$\kappa$-closed} if in $\P$ every infinite descending chain of length less than $\kappa$ has a lower bound.
\begin{theorem}\label{kclosedforcingpreservedck}
Let $M$ be a countable transitive model of $\ZFCUR$ + DC$_\kappa$-scheme, $\P \in M$ be such that $(\P \text{ is } \kappa^+\text{-closed})^M$ and $G$ be an $M$-generic filter over $\P$. Then $M[G] \models$ $\ZFCUR$ + DC$_\kappa$-scheme.
\end{theorem}
\begin{proof}
 For every $\alpha$-sequence $s$ of $\P$-names, let $\dot{s}^{(\alpha)}$ denote the canonical $\P$-name such that $\dot{s}^{(\alpha)}_G$ is an $\alpha$-sequence in $M[G]$ with $\dot{s}^{(\alpha)}_G(\eta) = s(\eta)_G$ for all $\eta < \alpha$. Given a $p \in \P$ and a suitable formula $\varphi$, a $\kappa$-sequence of the form $\<\<p_\alpha, \dot{x}_\alpha> : \alpha < \kappa>$, where $\<p_\alpha, \dot{x}_\alpha> \in \P \times M^\P$, is said to be a \textit{ $\varphi$-chain below} $p$ if $\<p_\alpha : \alpha < \kappa>$ is a descending chain below $p$ and for every $\alpha < \kappa$, $p_\alpha \forces \varphi(\dot{s}^{(\alpha)}, \dot{x}_{\alpha+1})$ where $s = \<\dot{x}_\eta : \eta < \alpha>$.

Suppose that $M[G] \models \forall x \exists y \varphi(x, y, u)$. There is some $p \in G$ such that $p \forces \forall x \exists y \varphi (x, y, \dot{u})$. Let $D$ be the set of forcing conditions that are a lower bound of some $\varphi$-chain below $p$. 
\begin{claim}
$D$ is dense below $p$.
\end{claim}
\begin{claimproof}
Fix some $r \leq p$. Let $\psi(x, y)$ be the following formula (with parameters $r, \P, \kappa$ and $\dot{u}$).
\begin{itemize}
    \item [] $\psi(x, y) =_{df}$ if $x$ is some $\<\<p_\eta, \dot{x}_\eta> : \eta < \alpha> \in (\P \times M^\P)^\alpha$, where  $\<p_\eta : \eta < \alpha>$ is a descending chain and $\alpha < \kappa$, then $y$ is some $\<q, \dot{x}> \in \P \times M^\P$ such that $q$ bounds $\<p_\eta : \eta < \alpha>$ and $q \forces \varphi(\dot{s}^{(\alpha)}, \dot{x}, \dot{u})$, where  $s = \<\dot{x}_\eta : \eta < \alpha>$.
\end{itemize}
Let $\P\downarrow r$ denote the set of conditions in $\P$ below $r$. In $M$, for every $x \in (\P \downarrow r \times M^\P)^{<\kappa}$, there is some $y \in \P \downarrow r \times M^\P$ such that $\psi(x, y)$ since $\P$ is $\kappa$-closed. By the DC$_\kappa$-scheme in $M$, there exists a $\varphi$-chain $\<\<p_\alpha, \dot{x}_\alpha> : \alpha < \kappa>$, where $\<p_\alpha : \alpha < \kappa>$ is below $r$ and hence below $p$. $\P$ is $\kappa^+$-closed, so there is some $q$ that bounds this $\varphi$-chain below $p$. Thus, $D$ is dense below $p$.\end{claimproof}

\noindent So there is some $q \in G$ that bounds a $\varphi$-chain, $\<\<p_\alpha, \dot{x}_\alpha> : \alpha < \kappa>$, below $p$. Let $s = \<\dot{x}_\alpha : \alpha < \kappa>$ and $f = \dot{s}^{(\kappa)}_G$. $f$ is then a $\kappa$-sequence in $M[G]$ and $\kappa$ is not collapsed in $M[G]$ as $\P$ is $\kappa$-closed. Moreover, $M[G] \models \varphi (f\restriction \alpha, f(\alpha), u)$ for all $\alpha < \kappa$ because $q \forces \varphi(\dot{s}^{(\alpha)}, \dot{x}_\alpha, \dot{u})$. 
\end{proof}
For any infinite cardinals $\kappa$ and $\lambda$ with $\kappa < \lambda$, $\textup{Col}(\kappa, \lambda)$ is the forcing poset consisting of all partial functions from $\kappa$ to $\lambda$ whose domain has size less than $\kappa$ (ordered by reverse inclusion).
\begin{theorem}\label{thm:ForcingDesotryDCK}
Forcing over countable transitive models of $\ZFCUR$ + Collection does not preserve the DC$_{\omega{_1}}$-scheme.
\end{theorem}
\begin{proof}
Let $M$ be a countable transitive model that satisfies $\ZFCUR$ + ``Every $A \subseteq \A$ has a tail of size $\omega_1$''. To have a model of this sort, we can start with a countable transitive model $N$ of $\ZFCUR$ in which $\A \sim \omega_2$. In $N$, let $\I$ be the $\A$-ideal of all sets of urelements of size $\omega_1$, and let $M$ be $N^\I$ as in Definition \ref{normalideal}.

By Theorem \ref{tail->collection} and Lemma \ref{Tailkappa->DCkappa}, both Collection and the DC$_{\omega_1}$-scheme hold in $M$.  Let $\P = \textup{Col}(\omega, \omega_1^M)$ and $G$ be $M$-generic over $\P$. In $M[G]$, every set of urelements is countable, because by Lemma \ref{McoversM[G]} (8) every $A \in M[G]$ is a subset of some $A' \in M$ such that $|A'| \leq \omega_1{^M}$ but $\omega_1{^M}$ is collapsed to $\omega$ in $M[G]$. $M[G]$ still has a proper class of urelements, so the DC$_{\omega_1}$-scheme fails in $M[G]$ by Lemma \ref{lemma:DCK->KRealized}.
\end{proof}
\noindent Since $\ZFCUR$ + Collection proves the DC$_\omega$-scheme, forcing over $\ZFCUR$ + Collection preserves DC$_\omega$-scheme as it preserves Collection.
\begin{question}
Does forcing over $\ZFCUR$ preserve the DC$_\omega$-scheme?
\end{question}

Forcing can also recover axioms. Next, we show that Collection and RP are \textit{necessarily forceble} when the ground model satisfies $\ZFCUR$ + DC$_\omega$-scheme (note that the DC$_\omega$-scheme does not imply Collection or RP by Theorem \ref{maintheorem1}).
\begin{theorem}\label{recovercollection}
Every forcing extension of a countable transitive model $M$ of $\ZFCUR$ + DC$_\omega$-scheme has a forcing extension that satisfies the Collection and Reflection Principle.
\end{theorem}
\begin{proof}
Let $M[G]$ be an arbitrary forcing extension of $M$. We may assume that $M[G] \models$ ``$\A$ is a proper class'' + $\neg$ Plenitude, since otherwise Collection holds in every forcing extension of $M[G]$ by Diagram \ref{ZFCUdiagram} and Lemma \ref{forcingpreservesZFCU}. So $M \models$ ``$\A$ is a proper class'' by Lemma \ref{McoversM[G]} (8) and in $M[G]$ there is a least infinite cardinal $\kappa$ that is not realized. Let $H$ be an $M[G]$-generic filter over $\textup{Col}(\omega, \kappa)$. As $\kappa$ is collapsed to $\omega$ in $M[G][H]$, every set of urelements in $M[G][H]$ is countable. In $M$, by the DC$_\omega$-scheme, for every $A \subseteq \A$ there is an infinite $B \subseteq \A$ that is disjoint from $A$; by Lemma \ref{McoversM[G]} (8), this fact is preserved by forcing so it holds in $M[G][H]$. This shows that $M[G][H] \models $ Tail. Therefore, $M[G][H] \models $ Collection $\land$ RP by Diagram \ref{ZFCUdiagram}.
\end{proof}
\noindent The assumption that $M\models$ DC$_\omega$-scheme in Theorem \ref{recovercollection} cannot be dropped: if $M$ has a proper class of urelements but every set of them is finite, by Lemma \ref{McoversM[G]} (8) and (9) this will remain the case in every forcing extension of $M$.

\subsection{Ground model definability}\label{section3.6}

Laver \cite{Laver2007-LAVCVL} and Woodin \cite{woodin2011continuum} proved independently the ground model definability for ZFC: every transitive model of ZFC is definable with parameters in all of its generic extensions. Laver's argument (which is also attributed to Hamkins \cite{Hamkins2003ExtensionsWT}) can be easily modified to show that every transitive model of $\ZFCUR$ with only a set of urelements is definable in all of its generic extensions with parameters (\cite[Theorem 85]{yao2023set}). And as a corollary, if $M$ is a transitive model of $\ZFCUR$ in which some cardinal $\kappa$ is not realized, then $M$ is definable in all of its generic extensions produced by $\kappa$-closed forcing notions (\cite[Corollary 85.1]{yao2023set}). Here we show how the ground model definability may \textit{fail} if the ground model has a proper class of urelements.

For any infinite set of $x$, $\textup{Fn}(x, 2)$ is the forcing poset consisting of all finite partial functions from $x$ to $2$ ordered by reversed inclusion, which adds a new subset to every set that is equinumerous with $x$.
\begin{theorem}
Let $M$ be a countable transitive model of $\ZFCUR$.
\begin{enumerate}
    \item If $M \models$ DC$_\omega$-scheme + ``$\A$ is a proper class'', then $M$ has a forcing extension in which $M$ is not definable with parameters.
    \item If $M \models $ Plentitude, then $M$ is not definable in any of its non-trivial forcing extensions. 
\end{enumerate}
\end{theorem}
\begin{proof}
(1) Suppose that $M \models$ DC$_\omega$-scheme + ``$\A$ is a proper class''. Let $\P \in M$ be $\textup{Fn}(\omega, 2)$ and $G$ be an $M$-generic filter over $\P$. Suppose \textit{for reductio} that $M$ is definable in $M[G]$ with a parameter $\dot{u}_G \in M[G]$ such that $M = \{ x \in M[G] : M[G] \models \varphi (x,\dot{u}_G)\}.$ Let $B' \in M$ be an infinite set of urelements disjoint from $ker(\dot{u})$, which exists by the DC$_\omega$-scheme. Then $M[G]$ contains a new countable subset $B$ of $B'$ which is not in $M$. Fix another countable set of urelements $C \in M$ that is disjoint from $ker(\dot{u}) \cup B'$. In $M[G]$, there will be an automorphism that swaps $C$ and $B$ while point-wise fixing $ker(\dot{u})$. Since $M[G] \models \neg \varphi (B,\dot{u}_G)$ and $ker(\dot{u}_G) \subseteq ker(\dot{u})$, it follows that $M[G] \models \neg \varphi (C,\dot{u}_G)$ and hence $C \notin M$, which is a contradiction.

(2) Suppose that $M \models$ Plentitude and consider any $M[G]$ such that $M \subsetneq M[G]$. First observe that there must be some set of urelements $B$ such that $B \in M[G] - M$. Fix some $\dot{x}_G \in M[G] - M$ of the least rank so that $\dot{x}_G \subseteq M$. Let $A = ker(\dot{x})$. It follows that $\dot{x}_G \subseteq V_\alpha (A)^M$ for some $\alpha$. By Plenitude and AC in $M$, there is a bijection $f$ from $V_\alpha (A)^M$ to a set of urelements,  so $f\restriction\dotx_G$ will produce a new set of urelements in $M[G]$.

For \textit{reductio}, suppose that $M = \{ x \in M[G] : M[G] \models \varphi(x,  \dot{u}_G)\}$ for some formula $\varphi$ with parameter $ \dot{u}_G$. Fix some set of urelements $B \in M[G] - M$ and $B' \in M$ such that $B \subseteq B'$. In $M$, by Plenitude $B'$ has a duplicate $E$ that is disjoint from $ker(\dot{u})$. Then $E$ has a new subset $D$ in $M[G]$ that is disjoint from $ker(\dot{u})$. By AC and Plenitude in $M$, we can again find a duplicate $C \in M$ of $D$ that is disjoint from $ker(\dot{u})$. So there will be an automorphism in $M[G]$ that swaps $C$ and $D$ while point-wise fixing $ker(\dot{u})$. As $M[G] \models \neg \varphi (D, \dot{u}_G)$, it follows that $M[G] \models \neg \varphi(C, \dot{u}_G)$ and hence $C \notin M$, which is a contradiction. 
\end{proof}

\appendix
\setcounter{secnumdepth}{0}
\section{Appendix}
In this appendix, we prove that the two forcing methods produce the same generic extensions. One can prove directly that $M[G]_{\scaleto{\#}{4pt}}$ satisfies $\ZFUR$ and then use the minimality of $M[G]$ and $M[G]_{\scaleto{\#}{4pt}}$ to conclude that they are the same model. Here, we take a different approach by analyzing the relationship between the names in $M^\P_{\scaleto{\#}{4pt}}$ and $M^\P$. In the following, $M$ is a countable transitive model of $\ZFUR$, $\P \in M$ a forcing poset, and $G$ an $M$-generic filter over $\P$.

To start with, there is a natural embedding from $M^\P_{\scaleto{\#}{4pt}}$ to $M^\P$ defined as follows.
\begin{definition}
Define $j : M^\P_{\scaleto{\#}{4pt}} \rightarrow M^\P$ by recursion as follows. For every $\tau \in M^\P_{\scaleto{\#}{4pt}}$,
\begin{equation*}
   j(\tau) =
    \begin{cases*}
     \{\<\tau, 1_\P>\} & if $\A(\tau)$  \\
     \{\<j(\sigma), p> : \<\sigma, p> \in \tau\}        & otherwise 
    \end{cases*}
  \end{equation*}
\end{definition}
\noindent We shall use Greek letters $\tau, \sigma, ...$ to denote the names in $M^\P_{\scaleto{\#}{4pt}}$ (Definition \ref{oldpnames}). A straightforward induction will show that $j$ is a 1-1 function from $M^\P_{\scaleto{\#}{4pt}}$ to  $M^\P$; in particular, the incompatibility condition in Definition \ref{newpnames} is trivially satisfied because each $j(\sigma)$ is in $M^\P$.

\begin{lemma}\label{tilde1-1}
For every $\sigma, \tau \in M^\P_{\scaleto{\#}{4pt}} $, $\sigma_G = \tau_G$ if and only if $j(\sigma)_{G} = j(\tau)_{G}$.
\end{lemma}
\begin{proof}
Note that the $G$-valuation is defined differently for names in $M^\P_{\scaleto{\#}{4pt}}$ and $M^\P$, but this should not cause any confusion. The proof is by a straightforward induction on the rank of $\sigma$ and $\tau$, so we omit it.\end{proof}

Therefore, the map $\sigma_G \mapsto j(\sigma)_G$ is a well-defined embedding from $M[G]_{\scaleto{\#}{4pt}}$ into $M[G]$. Next, we show that this embedding is elementary. The key observation here is that every name $\dotx$ in $M^\P$ has a \textit{set-counterpart}, $\dotx^{\textup{Set}}$, such that $\dotx_G = \dotx^{\textup{Set}}_G$ whenever $\dotx_G$ is a set.

\begin{definition}
For every $\dotx \in M^\P$, define by recursion
\begin{align*}
\xset  =& \{ \<\yset, s> \mid \doty \in M^\P \land \exists p\in \P (\<\doty, p> \in \dotx \land s \leq p) \land \forall r \in \P, a \in \A (\<a, r> \in \doty \to s \bot r) \}\\ 
      &\cup \{\<\check{a}, s> \mid \exists p, r \in \P, \doty \in M^\P, a \in \A (\<\doty, p> \in \dotx \land \<a, r> \in \doty \land s \leq p \land s \leq r)\}
\end{align*}
\end{definition}
\noindent The idea is that we forget about the urelements in the domain of $\dotx$ and then  for each $\P$-name $\doty$ in $dom(\dotx)$, we pair $\yset$ with some suitable conditions depending on whether $\doty$ is treated as a urelement or a set. 

\begin{lemma}\label{lemma:XSetinRan(j)}
For each $\dotx \in M^\P$, $\xset = j(\sigma)$ for some $\sigma \in M^\P_{\scaleto{\#}{4pt}}$.
\end{lemma}
\begin{proof}
By an induction on the rank of $\dotx$. If $dom(\dotx) \subseteq \A$, then $\xset = j(\emptyset)$. Note that for every urelement $a$, $\check{a} = j(a)$. So suppose that the lemma holds for every $\doty \in M^\P \cap \dotx$. Define
$$\sigma = \{\<j^{-1}(\doty), s> \mid \<\doty, s> \in \xset\}.$$
Then $\xset = j(\sigma)$.\end{proof}

\begin{lemma}\label{lemma:SetCounterpart}
For every $\dotx \in M^\P$, $\xset_G \subseteq \dotx_G$ and $\dotx_G \subseteq \xset_G$.
\end{lemma}
\begin{proof}
Suppose that the lemma holds for every $\doty \in M^\P \cap \dotx$. 

We first show that $\xset_G \subseteq \dotx_G$. Let $z \in \xset_G$.

\textit{Case 1.} $z$ is a set. Then $z = \yset_G$, where $\<\yset, s> \in \xset$ with some $s \in G$. This means that $\<\doty, p> \in \dotx$ for some $p \geq s$, and so $\doty_G \in \dotx_G$. Moreover, if $\<a, r> \in \doty$ for any urelement $a$, then $s$ and $r$ are incomptabile, which means $\doty_G$ must be a set. By the induction hypothesis, we have $\yset_G$ and $\doty_G$ are co-extensional. Therefore, $z = \doty_G \in \dotx_G$.

\textit{Case 2.} $z = a$ for some urelement $a$. Note that $\yset_G$ is always a set for every $\doty$. So $\<\check{a}, s> \in \xset$ and $s \in G$. By the constructin of $\xset$, it follows that $a \in \dotx_G$. Therefore, $\xset_G \subseteq \dotx_G$.

Next, we show that  $\xset_G \subseteq \dotx_G$. Let $\doty_G \in \dotx_G$, where $\<\doty, p>$ for some $p \in G$.

\textit{Case 1.} $\doty_G = a$ for some urelement $a$. Then $\<a, r> \in \doty$ for some $r \in G$. So there is some $s \in G$ with $s \leq p, r$. Thus, $\<\check{a}, s> \in \xset$ and hence $a \in \xset_G$.

\textit{Case 2.} $\doty_G$ is a set. By the induction hypothesis, we have $\doty_G = \yset_G$. So it remains to show that $\yset_G \in \xset_G$. Define
\begin{align*}
D_1 =&\{s \in \P \mid \exists a \in \A, r \in \P (\<a, r> \in \doty \land s \leq r)\};\\ 
 D_2 =&\{s \mid \forall a\in \A, r \in P (\<a, r> \in \doty \to r \bot s)\}.
\end{align*}
$D_1 \cup D_2$ is dense below $p$: if $q \leq p$ and $q \notin D_2$, then $q$ is comptable with some $r$ such that $\<a, r> \in \doty$ and so there will be some $s \in D_1$ with $s \leq q$. Thus, $(D_1 \cup D_2)\cap G$ is nonempty. But $D_1 \cap G$ must be empty, otherwise $\doty_G$ will be a urelement, which contradicts our case assumption. It follows that there is some $s \in D_2 \cap G$ such that $s \leq p$, making $\<\yset, s> \in \xset$. So we have $\yset_G \in \xset_G$. Therefore,  $\dotx_G \subseteq \xset_G$. This completes the proof. \end{proof}

\begin{theorem}
$M[G] = M[G]_{\scaleto{\#}{4pt}}$.
\end{theorem}
\begin{proof}
If $ M[G]_{\scaleto{\#}{4pt}} \models \ZFUR$, then $ M[G]_{\scaleto{\#}{4pt}}$ will be the least transitive model of $\ZFUR$ which extends $M$ and contains $G$ so we will have $M[G]_{\scaleto{\#}{4pt}} = M[G]$ by Theorem \ref{fundamentalthmofforcing} and Lemma \ref{McoversM[G]}. Thus, it suffices to show that the map $\sigma_G \mapsto j(\sigma)_G$ is an elementary embedding from $M[G]_{\scaleto{\#}{4pt}}$ to $M[G]$, and we prove this by induction on formulas. Atomic and Boolean cases are immediate, so it remains to show that for every $\sigma_1, ..., \sigma_n \in M^\P_{\scaleto{\#}{4pt}}$, 
$$M[G]_{\scaleto{\#}{4pt}} \models  \exists x \varphi (x, \sigma_{1_G}, ..., \sigma_{n_G}) \Leftrightarrow M[G] \models  \exists x \varphi (x, j(\sigma_1)_G, ..., j(\sigma_n)_G).$$ 
$\Rightarrow$ holds by the induction hypothesis. Suppose that $M[G] \models \varphi(\dotx_G, j(\sigma_1)_G, ..., j(\sigma_n)_G)$ for some $\dotx \in M^\P$. If $\dotx_G$ is some urelement $a$, then $\dotx_G = j(a)_G$; otherwise, $\dotx_G = \xset_G$ by Lemma \ref{lemma:SetCounterpart} and by Lemma \ref{lemma:XSetinRan(j)} $\xset = j(\sigma)$ for some $\sigma \in M^\P_{\scaleto{\#}{4pt}}$. Therefore, $\dotx_G = j(\sigma)_G$ for some $\sigma \in M^\P_{\scaleto{\#}{4pt}}$ and so $M[G]_{\scaleto{\#}{4pt}} \models  \exists x \varphi (x, \sigma_{1_G}, ..., \sigma_{n_G})$ by the induction hypothesis.\end{proof}

\printbibliography

\end{document}